\documentclass[a4paper,11pt,reqno]{amsart}

\usepackage[english]{babel}
\usepackage{amsmath}
\usepackage{amssymb}
\usepackage{amsfonts}
\usepackage{amsthm}
\usepackage{float}
\usepackage{stmaryrd}
\usepackage{graphics}
\usepackage{graphicx}
\usepackage{subfig}
\usepackage{datetime}
\usepackage{color}

\newcommand{\Hloc}{H^1_{\textnormal{loc}}}
\newcommand{\eps}{\varepsilon}
\renewcommand{\d}{\,\mathrm{d}}

\newcommand{\mc}{\mathcal}
\newcommand{\diff}{\!\setminus\!}
\newcommand{\ellz}{{\ell_0}}
\newcommand{\ellzk}{{\ell_0^k}}
\newcommand{\nuk}{{\nu^k}}
\newcommand{\vz}{{v_0}}
\newcommand{\vzk}{{v_0^k}}
\newcommand{\vu}{{v_1}}
\newcommand{\vuk}{{v_1^k}}
\newcommand{\vzd}{{\dot{v}_0}}
\newcommand{\vzdk}{{\dot{v}_0^k}}
\newcommand{\kappak}{{\kappa^k}}

\newcommand{\zk}{{z^k}}
\newcommand{\zdk}{{\dot{z}^k}}
\newcommand{\vk}{{v^k}}
\newcommand{\vkt}{{v^k_t}}
\newcommand{\vkx}{{v^k_x}}
\newcommand{\ellk}{{\ell^k}}
\newcommand{\lambdad}{{\dot{\lambda}}}
\newcommand{\lambdak}{{\lambda^k}}
\newcommand{\lambdadk}{{\dot{\lambda}^k}}
\newcommand{\elld}{{\dot{\ell}}}
\newcommand{\elldk}{{\dot{\ell}^k}}
\newcommand{\Ak}{{A^k}}
\newcommand{\Akt}{{A^k_t}}
\newcommand{\Akx}{{A^k_x}}
\newcommand{\Hk}{{H^k}}
\newcommand{\Hkt}{{H^k_t}}

\newcommand{\varphik}{{\varphi^k}}
\newcommand{\psik}{{\psi^k}}

\newcommand{\omk}{{\omega^k}}
\newcommand{\omdk}{{\dot{\omega}^k}}
\newcommand{\omd}{{\dot{\omega}}}
\def\enne{\mathbb{N}}
\def\zeta{\mathbb{Y}}

\def\erre{\mathbb{R}}

\renewcommand{\to}{\rightarrow}

\numberwithin{equation}{section}
\newtheorem{thm}{Theorem}[section]
\newtheorem{defi}[thm]{Definition}
\newtheorem{prop}[thm]{Proposition}
\newtheorem{lemma}[thm]{Lemma}
\newtheorem{cor}[thm]{Corollary}

\theoremstyle{definition}
\begingroup
\newtheorem{rmk}[thm]{Remark}

\endgroup

\theoremstyle{remark}
\begingroup
\endgroup

\begin{document}
	
	\author{Filippo Riva}
	
	\title[Continuous dependence for a 1d debonding model]{A continuous dependence result\\for a dynamic debonding model in dimension one}
	
	\begin{abstract}
		In this paper we address the problem of continuous dependence on initial and boundary data for a one-dimensional dynamic debonding model describing a thin film peeled away from a substrate. The system underlying the process couples the (weakly damped) wave equation with a Griffith's criterion which rules the evolution of the debonded region. We show that under general convergence assumptions on the data the corresponding solutions converge to the limit one with respect to different natural topologies.
	\end{abstract}

	\maketitle
		
	{\small
		\par
	}

	\pagenumbering{arabic}
	
\medskip

\tableofcontents

	\section*{Introduction}	
	The interest of the physical and engineering community on dynamic debonding models involving one spatial dimension originates in the '70s from the works of Hellan \cite{Hela, Helb, Helbook}, Burridge \& Keller \cite{BurrKell} and carries on in the '90s with the ones of Freund collected in \cite{Fre90}. The importance of this kind of models relies on the fact that they possess deep similarities to the theory of dynamic crack growth based on Griffith's criterion, but at the same time they are much easier to treat, allowing an exhaustive comprehension of the involved physical processes. More recently dynamic debonding models have been resumed by several authors, see for instance Dumouchel and others \cite{DouMarCha07, DouMarCha08, LBDM12}, but only in the last few years a rigorous mathematical formulation has been adopted: we are referring to \cite{DMLazNar16, LazNarkappa, LazNarinitiation, RivNar}, in which existence and uniqueness results are stated, or to \cite{LazNarkappa, LazNar}, where the so-called quasistatic limit problem is addressed. It concerns whether or not dynamic evolutions converge to quasistatic ones (see \cite{MieRou15} for the general discussion about quasistatic or rate-independent processes) when inertia goes to zero. We also refer to \cite{Buc} and \cite{MadTom} for adhesion and debonding problems in the static and quasistatic regime.\par 
	Nevertheless we are not aware of the presence in literature of continuous dependence results for debonding models, despite the importance of the issue and despite partial achievements in this direction have already been obtained in the more complicated framework of Fracture Dynamics, see for instance \cite{Cap, DMLuc}. Therefore the aim of our paper is filling this gap, giving a positive answer to the question of continuous dependence in a general version of the one-dimensional dynamic debonding model. The result will be used in a forthcoming paper to deal with the crucial problem arising in Mechanics of the quasistatic limit in this context.\par 
	To describe the model we are going to analyse let us consider a perfectly flexible and inextensible thin film partially glued to a flat rigid substrate. In an orthogonal coordinate system $(x,y,z)$, in which the substrate is identified with the half plane $\{(x,y,z)\mid x\ge 0,\, z=0 \}$, we assume the deformation of the film at time $t\ge 0$ is parametrized by $(x,y,0)\mapsto(x+h(t,x),y,u(t,x))$, where the scalar functions $h$ and $u$ represent the horizontal and the vertical displacement, respectively. Since the second component $y$ is assumed to be constant it will be ignored in the rest of the paper; this means that the debonding process takes place in the vertical half plane $\{(x,z)\mid x\ge 0 \}$. At every time $t\ge 0$ the debonded part of the film is the segment $\{(x,0)\mid x\in[0,\ell(t)) \}$, where $\ell$ is a nondecreasing function representing the debonding front. This in particular implies that the displacement $(h(t,x),u(t,x))$ is identically zero on the half line $\{(x,0)\mid x\ge \ell(t) \}$. As in \cite{DMLazNar16} and \cite{RivNar} in this work we make the crucial assumption that $\ellz:=\ell(0)>0$, namely at the initial time $t=0$ the film is already debonded in the segment $\{(x,0)\mid x\in[0,\ellz) \}$; see instead \cite{LazNarinitiation} for the analysis of the singular case in which initially the film is completely glued to the substrate. At the endpoint $x=0$ we finally prescribe a boundary condition for the vertical displacement $u(t,0)=w(t)$. By linear approximation, inextensibility of the film provides an explicit formula for the horizontal displacement:
	\begin{equation*}
		h(t,x)=\frac 12 \int_{x}^{\ell(t)}u_x^2(t,\xi)\d\xi.
	\end{equation*}
	The vertical displacement $u$ and the debonding front $\ell$ instead solve the system:	
	\begin{subequations}\label{coupled}
		\begin{align}
			\begin{split}\label{problemu}
			&\begin{cases}
			u_{tt}(t,x)-u_{xx}(t,x)+\nu u_t(t,x)=0, \quad& t > 0 \,,\, 0<x<\ell(t),  \\
			u(t,0)=w(t), &t>0, \\
			u(t,\ell(t))=0,& t>0,\\
			u(0,x)=u_0(x),\quad&0<x<\ell_0,\\
			u_t(0,x)=u_1(x),&0<x<\ell_0,		
			\end{cases}
			\end{split}\\
			\begin{split}\nonumber
			&
			\end{split}	\\
			\begin{split}\label{energycrit}
			&+\mbox{Energy criteria satisfied by }u\mbox{ and }\ell,
			\end{split}
		\end{align}
	\end{subequations}	
	where the initial conditions $u_0$ and $u_1$ are given functions, and the parameter $\nu\ge 0$ takes into account the friction produced by air resistance.\par 
	The paper is organised as follows: in Section \ref{sec1} we first give a rigorous mathematical presentation of the debonding model and we introduce the energy criteria appearing in \eqref{energycrit} that the pair $(u,\ell)$ has to satisfy (see Griffith's criterion \eqref{Griffithcrit}). We then state the result of existence and uniqueness for solutions to problem \eqref{coupled} proved in \cite{RivNar}. Finally we present the continuous dependence problem: we consider sequences of data converging in the natural topologies to some limit data, see \eqref{conv}, and we wonder whether and in which sense the sequence of solutions to \eqref{coupled} corresponding to these sequences of data, denoted by $\{(u^k, \ellk)\}_{k\in\enne}$, converges to the solution corresponding to the limit ones, denoted by $(u,\ell)$.\par 
	Section \ref{sec2} is devoted to the analysis of the convergence of the sequence of vertical displacements $\{u^k\}_{k\in\enne}$ assuming a priori that the sequence of debonding fronts $\{\ellk\}_{k\in\enne}$ converges to $\ell$ in some suitable topology. The main outcomes of this Section are collected in \eqref{allconv}, see also Remark~\ref{auxiliar}. This is, however, a continuous dependence result for problem \eqref{problemu}, still not coupled with \eqref{energycrit}, see Remark~\ref{Remcont}.\par 
	In Section \ref{sec3} we finally state and prove our continuous dependence result for the coupled problem, see Theorem~\ref{finalthm}, showing that the convergence of the sequence of debonding fronts we postulated in Section \ref{sec2} actually happens. The strategy of the proof strongly relies on a representation formula for solutions to \eqref{problemu} proved in \cite{RivNar}, see \eqref{repformula} and \eqref{lambda}. Furthermore the argument exploits the idea used in \cite{RivNar} that a certain operator is a contraction with respect to a suitable distance, see \eqref{distance} and Propositions~\ref{estimatelprop} and \ref{estimatevprop}.
	
	\section*{Notations}
	In this Section we collect some notation and some definition that we will use several times during the paper. They have already been introduced and used in \cite{DMLazNar16} and \cite{RivNar}, so we refer to them for a wide and more complete explanation.
	\begin{rmk}
		Throughout the paper every function in $W^{1,p}(a,b)$, for $-\infty<a<b<+\infty$ and $p\in[1,+\infty]$, is always identified with its continuous representative on $[a,b]$.\par 
		Furthermore the derivative of any function of real variable is always denoted by a dot (i.e. $\dot{f}$, $\dot{\ell}$, $\dot{\varphi}$, $\dot{v}_0$), regardless of whether it is a time or a spatial derivative.
	\end{rmk}
\textbf{Geometric considerations.} Fix $\ell_0>0$ and consider a function $\ell \colon [0,+\infty)\to [\ellz,+\infty)$, which will play the role of the debonding front, satisfying:
	\begin{subequations}\label{elle}
		\begin{equation}\label{ellea}
		\ell\in C^{0,1}([0,+\infty)) ,
		\end{equation}
		\begin{equation}\label{elleb}
		\ell(0)=\ell_0\mbox{ and } 0\le\dot\ell(t)< 1 \mbox{ for a.e. }t\in [0,+\infty).
		\end{equation}
	\end{subequations}	
	Given such a function we define the sets:
	\begin{align*}
	&\Omega := \{ (t,x)\mid t>0\,,\, 0 < x < \ell(t)\},\\
	&\Omega'_1 \,:=\{ (t,x)\in\Omega\mid t\le x\mbox{ and }t+x\le \ell_0\},\\		
	&\Omega'_2 \,:=\{ (t,x)\in\Omega\mid t>x\mbox{ and }t+x<\ell_0\},\\	
	&\Omega'_3 \,:=\{ (t,x)\in\Omega\mid t<x\mbox{ and }t+x>\ell_0\},\\
	&\Omega'\,\, :=	\,\,\Omega'_1\cup\Omega'_2\cup\Omega'_3	,\\
	&\Omega_T :=\{ (t,x)\in\Omega\mid t<T\},\\
	&\Omega'_T:=\{ (t,x)\in\Omega'\mid t<T\},\\
	&\!(\Omega'_i)_T:=\{ (t,x)\in\Omega'_i\mid t<T\},\mbox{ for }i=1,2,3,
	\end{align*}
	Moreover, for $t\in[0,+\infty)$, we introduce the functions:
	\begin{equation}\label{phipsidef}
	\varphi(t):= t {-} \ell(t) \,\mbox{, }\quad \psi(t):=t{+}\ell(t),
	\end{equation} 
	and we define:
	\begin{equation}\label{omegadef}
	\omega\colon [\ell_0,+\infty) \to [-\ell_0,+\infty) , \quad
	\omega(t):=\varphi\circ\psi^{-1}(t).
	\end{equation}
	\begin{rmk}
		By \eqref{elleb} $\psi$ turns out to be a bilipschitz function ($1\le \dot{\psi}< 2$), while $\varphi$ turns out to be Lipschitz with $0<\dot\varphi(t)\le 1$ for a.e. $t\in[0,+\infty)$. Hence $\varphi$ is invertible with absolutely continuous inverse. As a byproduct we get that $\omega$ is Lipschitz too and for a.e. $t\in[\ell_0,+\infty)$  it holds true:
		\begin{equation}\label{omegadot}
		0<\dot\omega(t)=\frac{1-\dot\ell(\psi^{-1}(t))}{1+\dot\ell(\psi^{-1}(t))}\le 1.
		\end{equation}
		So $\omega$ is invertible with absolutely continuous inverse too.
	\end{rmk}\noindent
	For $(t,x)\in\Omega'$ we also introduce the set: 
		\begin{equation}\label{rettangoli}
	R(t,x)= \{(\tau,\sigma) \in \Omega' \mid 0 < \tau < t,\,\,\, \gamma_1(\tau;t,x) < \sigma < \gamma_2(\tau;t,x) \},
	\end{equation}
	where
	\begin{equation}\label{bordi}
	\begin{aligned}		
	&\gamma_1(\tau;t,x) = \begin{cases}
	x{-}t{+}\tau, & \mbox{if } (t,x)\in\Omega_1', \\
	|x{-}t{+}\tau|, & \mbox{if } (t,x)\in\Omega_2',\\
	x{-}t{+}\tau, & \mbox{if } (t,x)\in\Omega_3',
	\end{cases}\,\\
	&\gamma_2(\tau;t,x) = \begin{cases}
	x{+}t{-}\tau, & \mbox{if } (t,x)\in\Omega_1',\\
	x{+}t{-}\tau, &\mbox{if } (t,x)\in\Omega_2',\\
	\tau{-}\omega(t{+}x), & \mbox{if } (t,x)\in\Omega_3' \mbox{ and } \tau \le \psi^{-1}(t{+}x),\\
	x{+}t{-}\tau, &  \mbox{if } (t,x)\in\Omega_3' \mbox{ and } \tau > \psi^{-1}(t{+}x),
	\end{cases}
	\end{aligned}
	\end{equation}
	\begin{figure}
		\subfloat{\includegraphics[scale=.4]{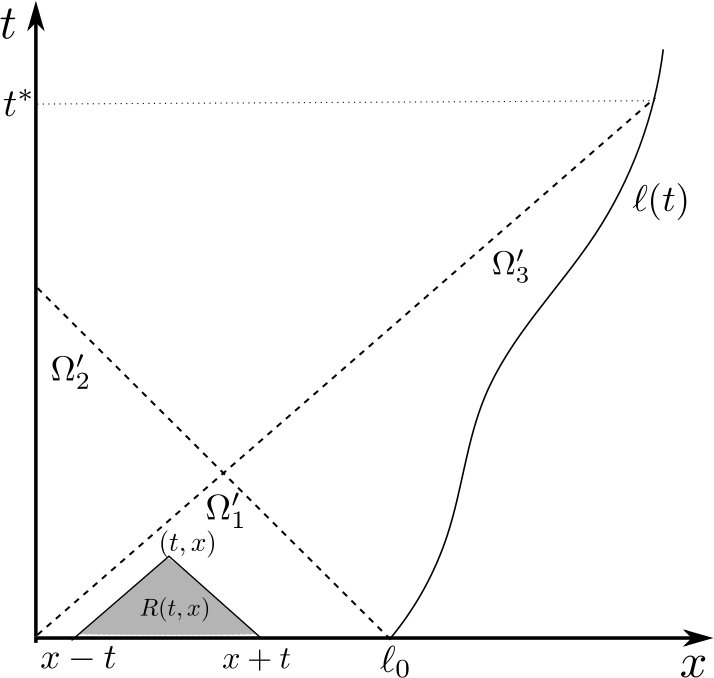}}\quad
		\subfloat{\includegraphics[scale=.4]{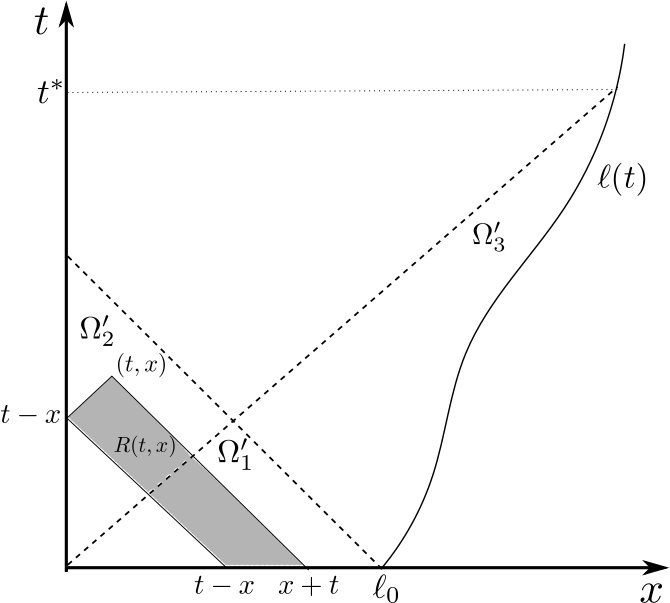}}\quad
		\subfloat{\includegraphics[scale=.4]{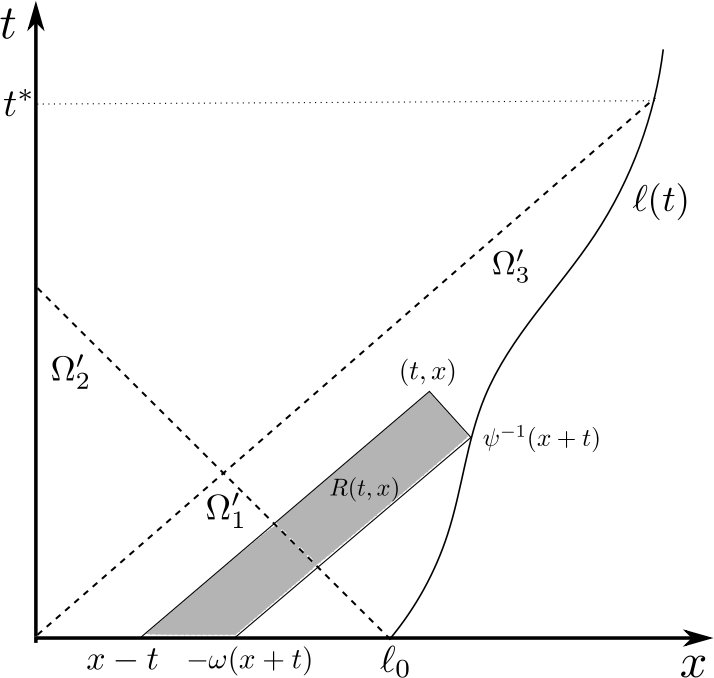}}
		\caption{The set $R(t,x)$ in the three possible cases.}\label{Figrettangoli}
	\end{figure}		
	are the left and the right boundary of $R(t,x)$, respectively. See Figure~\ref{Figrettangoli}.
	\begin{rmk}
		We warn the reader that, for the sake of clarity, during the whole paper we shall not write $\Omega_\ell$, $\Omega'_\ell$, $R_\ell(t,x)$, $\varphi_\ell$ or $\omega_\ell$, even if all of the sets and the functions introduced in this Section depend explicitely on the function $\ell$.
	\end{rmk}
	\textbf{Mathematical objects.} Finally let us define the spaces:
	\begin{gather*}
	\widetilde{H}{^1}(\Omega) := \{u \in \Hloc(\Omega)\mid u \in H^1(\Omega_T) \mbox{ for every } T>0\},\\
	\widetilde{H}{^1}(\Omega') := \{u \in \Hloc(\Omega')\mid u \in H^1(\Omega'_T) \mbox{ for every } T>0\},\\
	\widetilde H^1(0,+\infty) :=\{u \in \Hloc(0,+\infty) \mid u\in H^1(0,T) \, \text{ for every } T>0\},\\
	\widetilde C^{0,1} ([\ellz,+\infty)) := \{ u \in C^0([\ellz,+\infty)) \mid u \in C^{0,1}([\ellz,X]) \text{ for every } X>\ellz \} .
	\end{gather*}

	\section{Statement of the problem}\label{sec1}
	\subsection{The debonding model}
	In this Section we make the definition of solution to \eqref{coupled} precise. We fix $\nu\geq 0$, $\ellz>0$ and we assume that the boundary and initial data satisfy:
\begin{subequations}\label{bdryregularity}
	\begin{equation}	
	w \in \widetilde  H^1(0,+\infty),	
	\end{equation}
	\begin{equation}\label{bdry2}
	u_0 \in H^1(0,\ell_0) , \quad u_1 \in L^2(0,\ell_0) .
	\end{equation}
	\begin{equation}\label{compatibilitycond}
		u_0(0)=w(0),\quad u_0(\ellz)=0.
	\end{equation}
\end{subequations}	
	To fix the ideas let us assume for the moment that the debonding front $\ell \colon [0,+\infty) \to [\ell_0,+\infty)$ is assigned and it satisfies \eqref{elle}.
	\begin{defi}
		\label{sol}
		We say that a function $u \in \widetilde{H}{^1}(\Omega)$ (resp.\ in $H^1(\Omega_T)$) is a solution of \eqref{problemu} if $u_{tt}-u_{xx}+\nu u_{t}=0$  holds in the sense of distributions in $\Omega$ (resp.\ in $\Omega_T$), the boundary conditions are intended in the sense of traces and the initial conditions $u_0$ and $u_1$ are satisfied in the sense of $L^2(0,\ell_0)$ and $H^{-1}(0,\ell_0)$, respectively.
	\end{defi}
\begin{rmk}
	The definition is well posed, since for a solution $u\in H^1(\Omega_T)$ we have that $u_t$ and $u_x$  belong to $L^2(0,T;L^2(0,\ell_0))$; this implies that $u_t$ and $u_{xx}$ are in $L^2(0,T;H^{-1}(0,\ell_0))$ and so by the wave equation $u_{tt}\in L^2(0,T;H^{-1}(0,\ell_0))$. Therefore $u_t\in H^1(0,T;H^{-1}(0,\ell_0))\subseteq C^0([0,T];H^{-1}(0,\ell_0))$ (see also~\cite{DMLazNar16}).
\end{rmk}
	To establish the rules governing the evolution of the debonding front $\ell$ we need to introduce for $t\in[0,+\infty)$ the internal energy of a solution $u$:
	\begin{equation*}
	\mathcal{E}(t):=\frac 12 \int_{0}^{\ell(t)}\left(u_t^2(t,x)+u_x^2(t,x)\right)\d x,
	\end{equation*}
	the energy dissipated by the friction of air:
	\begin{equation*}
	\mathcal{A}(t):=\nu\int_{0}^{t}\int_{0}^{\ell(\tau)}u_t^2(\tau,\sigma)\d\sigma\d\tau,
	\end{equation*}
	and the work of the external loading:
	\begin{equation*}
		\mc W(t):=-\int_{0}^{t}\dot{w}(s)u_x(s,0)\d s.
	\end{equation*}
	\begin{rmk}
		As proved in \cite{RivNar}, the internal energy $\mc E(t)$ is well defined for every $t\in[0,+\infty)$ since $u$ turns out to be in $C^0([0,+\infty);H^1(0,+\infty))$ and in $ C^1([0,+\infty);L^2(0,+\infty))$; we present this result in Theorem~\ref{exuniq}. The expression $u_x(s,0)$ makes instead sense due to the representation formula for solutions to \eqref{problemu} introduced and used in \cite{RivNar}, see \eqref{vxpunto} and the subsequent discussion.
	\end{rmk}
	Moreover we assume that the glue between the substrate and the film behaves in a brittle fashion, thus the energy dissipated during the debonding process in the time interval $[0,t]$ is given by the formula
	\begin{equation}
	\int_{\ell_0}^{\ell(t)} \kappa(x) \d x,
	\end{equation}
	where $\kappa \colon [\ell_0,+\infty) \to (0,+\infty)$ is a measurable function representing the local toughness of the glue.\par 
	In our model we postulate that the debonding front $\ell$ has to evolve following two principles, which will replace the vague condition \eqref{energycrit}. The first one, called energy-dissipation balance, simply states that during the evolution the following equality between internal energy, dissipated energy and work of the external loading has to be satisfied:
	\begin{equation}
	\label{edb}
	\mc E(t)+\mc A(t)+\int_{\ell_0}^{\ell(t)} \kappa(x) \d x=\mc E(0)+\mc W(t),\quad\quad\text{for every }t\in[0,+\infty).
	\end{equation}
	The second one, called maximum dissipation principle, states that $\ell$ has to grow at the maximum speed which is consistent with the energy-dissipation balance (see also \cite{Lar10}):
	\begin{equation}
		\label{mdp}
		\dot \ell(t) = \max\{\alpha\in[0,1) \mid \kappa(\ell(t))\alpha = G_\alpha(t) \alpha\},\quad\quad\text{for a.e. } t\in[0,+\infty),
	\end{equation}
	where $G_\alpha(t)$ is the so-called dynamic energy release rate at speed $\alpha$, a quantity which measures the amount of energy spent by the debonding process. It is obtained as a sort of partial derivative of the total energy with respect to the elongation of the debonding front; we refer to \cite{DMLazNar16}, \cite{Fre90} or \cite{RivNar} for more details, since in this work we do not need its rigorous definition.\par 
	We only want to mention that in our context it has the expression:
	\begin{equation}\label{Galpha}
		G_\alpha(t)=\frac{1-\alpha}{1+\alpha}G_0(t),\quad \text{for a.e. }t\in[0,+\infty),
\end{equation}
	where $G_0$ can be explicitely written as:
	\begin{equation}\label{Gzero}
		G_0(t)=\frac 12 \left[\dot{u}_0(\ell(t){-}t)-u_1(\ell(t){-}t)+\nu\int_{0}^{t}u_t(\tau,\tau{-}t{+}\ell(t))\d\tau\right]^2,\quad \text{for a.e. }t\in\left[0,\frac{\ell_0}{2}\right),
\end{equation}
	and then extended to the whole $[0,+\infty)$ via a suitable procedure. It is also worth recalling that if $\alpha=\elld(t)$ one can write:
	\begin{equation}\label{Griffregular}
		G_{\elld(t)}(t)=\frac 12(1-\elld(t)^2)u_x(t,\ell(t))^2,\quad \text{for a.e. }t\in[0,+\infty).
\end{equation}
	In \cite{DMLazNar16} and \cite{RivNar} it has been shown that the two principles \eqref{edb} and \eqref{mdp} together are equivalent to the following system, called Griffith's criterion:
	\begin{equation}\label{Griffithcrit}
	\begin{cases}
	\,\,0\le\dot{\ell}(t)<1,\\ 
	\,\,G_{\dot\ell(t)}(t)\le \kappa(\ell(t)),\\ 
	\left[ G_{\dot\ell(t)}(t)-\kappa(\ell(t))\right]\dot{\ell}(t)=0, 
	\end{cases}\quad\quad\quad\quad \text{for a.e. } t\in[0,+\infty),
	\end{equation}
	which in turn is equivalent to an ordinary differential equations for the debonding front $\ell$:
	\begin{equation}\label{equation}
	\dot{\ell}(t)=\max\left\{\frac{G_0(t)-\kappa(\ell(t))}{G_0(t)+\kappa(\ell(t))},0\right\},\quad\quad\quad\quad \text{for a.e. } t\in[0,+\infty).
	\end{equation}
	\begin{rmk}
		The dynamic energy release rate $G_\alpha (t)$ depends on the solution $u$ of problem \eqref{problemu} and on the debonding front $\ell$ itself, as the reader can see from \eqref{Galpha} and \eqref{Gzero}, so equation \eqref{equation} only makes sense if coupled with problem \eqref{problemu}.
	\end{rmk}\noindent
	We are now in the position to give the following Definition:
	\begin{defi}\label{solcoup}
		Assume $\ell\colon[0,+\infty)\to[\ell_0,+\infty)$ satisfies \eqref{elle}; let $u\colon[0,+\infty)^2\to\erre$ be such that $u\in\widetilde{H}^1(\Omega)$ (resp. in $H^1(\Omega_T)$). We say that the pair $(u,\ell)$ is a solution of the coupled problem (resp. in $[0,T]$) if:
		\begin{itemize}
			\item[i)] $u$ solves problem \eqref{problemu} in $\Omega$ (resp. in $\Omega_T)$ in the sense of Definition~\ref{sol},
			\item[ii)] $u\equiv 0$ outside $\overline{\Omega}$ (resp. in $([0,T]\!\times\![0,+\infty))\setminus\overline{\Omega_T}$),
			\item[iii)]  $(u,\ell)$ satisfies Griffith's criterion \eqref{Griffithcrit} for a.e. $t\in[0,+\infty)$ (resp. for a.e. $t\in[0,T]$).
		\end{itemize}	
	\end{defi} \noindent
	In \cite{RivNar} it has been proved that under suitable assumptions on the toughness $\kappa$ coupled problem \eqref{problemu}\&\eqref{Griffithcrit} admits a unique solution. The result is the following:
		\begin{thm}\label{exuniq}
		Fix $\nu\ge 0$, $\ell_0>0$ and consider $u_0$, $u_1$ and $w$ satisfying \eqref{bdryregularity}. Assume that the measurable function $\kappa\colon[\ell_0,+\infty)\to(0,+\infty) $ fulfills the following property:
		\begin{equation}\label{kappaeps}
		\mbox{for every }x\in[\ell_0,+\infty)\text{ there exists }\varepsilon=\varepsilon(x)>0\text{ such that } \kappa\in C^{0,1}([x,x+\varepsilon]).
		\end{equation}
		Then there exists a unique pair $(u,\ell)$ solution of the coupled problem in the sense of Definition~\ref{solcoup}. Moreover $u$ has a continuous representative on $\overline{\Omega}$ and it holds:
		\begin{equation*}
		u\in C^0([0,+\infty);H^1(0,+\infty))\cap C^1([0,+\infty);L^2(0,+\infty)). 
		\end{equation*}
	\end{thm}
	The strategy of the proof relies in a representation formula (Duhamel's principle) valid for small times for the solution $u$ of \eqref{problemu} and for an auxiliary function $v$ defined as $v(t,x):=e^{\nu t/2}u(t,x)$. Since later on we will widely exploit it, we now want to say something more about this formula: to present it we first introduce the boundary and initial data of $v$, namely
	\begin{equation}\label{datav}
	\begin{gathered}
	z(t)=e^{\nu t/2}w(t),\\
	v_0(x)=u_0(x)\quad\text{and}\quad v_1(x)=u_1(x)+\frac{\nu}{2}u_0(x).
	\end{gathered}
	\end{equation}
	\begin{rmk}
		The functions $z$, $v_0$ and $v_1$ satisfy \eqref{bdryregularity} if and only if $w$, $u_0$ and $u_1$ do the same.
	\end{rmk}\noindent
	Then we recall that $v$ solves (in the sense of Definition~\ref{sol}) the following problem:
		\begin{equation}
		\label{problemv}
		\begin{cases}
		v_{tt}(t,x)-v_{xx}(t,x)-\displaystyle\frac{\nu^2}{4}v(t,x)=0, \quad& t > 0 \,,\, 0<x<\ell(t),  \\
		v(t,0)=z(t), &t>0, \\
		v(t,\ell(t))=0,& t>0,\\
		v(0,x)=v_0(x),\quad&0<x<\ell_0,\\
		v_t(0,x)=v_1(x),&0<x<\ell_0.
		\end{cases}
		\end{equation}	
	Thanks to the fact that $v$ solves \eqref{problemv}, in \cite{RivNar} it has been shown that, given $T<\frac{\ellz}{2}$, the pair $(u,\ell)$ is a solution of the coupled problem in $[0,T]$ if and only if the pair $(v,\ell)$ satisfies:
	\begin{equation}\label{repformula}
		\begin{cases}
		v(t,x)=A(t,x)+\displaystyle\frac{\nu^2}{8}\iint_{R(t,x)}v(\tau,\sigma)\d\sigma\d \tau, &\mbox{for every }(t,x)\in\overline{\Omega_T},\\
		\ell(t)=\ellz+\displaystyle\int_{0}^{t}\max\left\{\Gamma_{v,\ell}(s),0\right\}\d s,&\mbox{for every } t\in[0,T].
		\end{cases}
	\end{equation}
	where $R(t,x)$ is as in \eqref{rettangoli}, and the functions $A$ and $\Gamma_{v,\ell}$ are defined as follows:
	\begin{equation}\label{A}
	A(t,x)=\begin{cases}\displaystyle
	\frac 12 v_0(x{-}t)+\frac 12 v_0(x{+}t)+\frac 12 \int_{x{-}t}^{x{+}t}v_1(s) \d s , & \mbox{if } (t,x)\in\Omega_1',\\
	\displaystyle z(t{-}x)-\frac 12 v_0(t{-}x)+\frac 12 v_0(t{+}x)+\frac 12 \int_{t{-}x}^{t{+}x}v_1(s) \d s, & \mbox{if } (t,x)\in\Omega_2',\\
	\displaystyle\frac 12 v_0(x{-}t)-\frac 12 v_0({-}\omega(x{+}t))+\frac 12 \int_{x{-}t}^{{-}\omega(x{+}t)}v_1(s) \d s, & \mbox{if } (t,x)\in\Omega_3',
	\end{cases}
	\end{equation}
and 
	\begin{equation}\label{Gamma}
	\Gamma_{v,\ell} (t)=\frac{\left[\dot{v}_0(\ell(t){-}t)-v_1(\ell(t){-}t)-\frac{\nu^2}{4}\int_{0}^{t}v(\tau,\tau{-}t{+}\ell(t))\d\tau\right]^2-2e^{\nu t}\kappa(\ell(t))}{\left[\dot{v}_0(\ell(t){-}t)-v_1(\ell(t){-}t)-\frac{\nu^2}{4}\int_{0}^{t}v(\tau,\tau{-}t{+}\ell(t))\d\tau\right]^2+2e^{\nu t}\kappa(\ell(t))}.
	\end{equation}
\begin{rmk}
	The letter $A$ in \eqref{A} stands for D'Alembert, indeed it is the solution of the undamped wave equation with data $v_0$, $v_1$ and $z$ and it can be recovered by an adaptation to our time-varying domain setting of the classical D'Alembert formula. The expression for $\Gamma_{v,\ell}$ in \eqref{Gamma} is instead obtained from \eqref{equation} and from the expression of $G_0(t)$ \eqref{Gzero} by rewriting it in terms of the auxiliary function $v$; see \cite{RivNar}, Subsection 3.2.
\end{rmk}	
We want to recall that, as proved in \cite{RivNar}, Lemmas~1.10 and 1.11, the function $A$ and the integral term
\begin{equation}\label{Acca}
	H(t,x):=\iint_{R(t,x)}v(\tau,\sigma)\d\sigma\d \tau,
\end{equation}
are both continuous on $\overline{\Omega'}$, they belong to $\widetilde{H}^1(\Omega')$ and furthermore, setting them to be identically zero outside $\overline{\Omega}$, they belong to $C^0([0,\frac{\ell_0}{2}];H^1(0,+\infty))$ and to $ C^1([0,\frac{\ell_0}{2}];L^2(0,+\infty))$. Moreover explicit expressions for the partial derivatives of $H$, valid for every $t\in\left[0,\frac{\ell_0}{2}\right]$ and for a.e. $x\in(0,\ell(t))$, are:
		\begin{subequations}\label{Hder}
			\begin{equation}\label{Hdert}
			\!\!H_t(t,x)=\begin{cases}
			\displaystyle\int_{0}^{t}v(\tau,x{+}t{-}\tau)\d\tau+\displaystyle\int_{0}^{t}v(\tau,x{-}t{+}\tau)\d\tau,  &\Omega_1', \\
			\displaystyle\int_{0}^{t}v(\tau,x{+}t{-}\tau)\d\tau-	\displaystyle\int_{0}^{t{-}x}v(\tau,t{-}x{-}\tau)\d\tau+\displaystyle\int_{t{-}x}^{t}v(\tau,x{-}t{+}\tau)\d\tau, & \Omega_2',\\
			\displaystyle\int_{0}^{t}\!\!\!v(\tau,x{-}t{+}\tau)\d\tau{-}\dot\omega(x{+}t)\!\!\displaystyle\int_{0}^{\psi^{-1}(x{+}t)}\!\!\!\!\!\!\!\!\!\!\!\!\!\!\!\!\!\!\!\!\!v(\tau,\tau{-}\omega(x{+}t))\d\tau{+}\!\!\displaystyle\int_{\psi^{-1}(x{+}t)}^{t}\!\!\!\!\!\!\!\!\!\!\!\!\!\!\!\!\!\!\!\!v(\tau,x{+}t{-}\tau)\d\tau, &\Omega_3',
			\end{cases}
			\end{equation}
			\begin{equation}\label{Hderx}
			\!\!H_x(t,x)=\begin{cases}
			\displaystyle\int_{0}^{t}v(\tau,x{+}t{-}\tau)\d\tau-\displaystyle\int_{0}^{t}v(\tau,x{-}t{+}\tau)\d\tau, & \Omega_1', \\
			\displaystyle\int_{0}^{t}v(\tau,x{+}t{-}\tau)\d\tau+	\displaystyle\int_{0}^{t{-}x}v(\tau,t{-}x{-}\tau)\d\tau-\displaystyle\int_{t{-}x}^{t}v(\tau,x{-}t{+}\tau)\d\tau, & \Omega_2',\\
			\!\!-\!\!\displaystyle\int_{0}^{t}\!\!\!\!\!v(\tau,x{-}t{+}\tau)\d\tau\!-\!\dot\omega(x{+}t)\displaystyle\int_{0}^{\psi^{-1}(x{+}t)}\!\!\!\!\!\!\!\!\!\!\!\!\!\!\!\!\!\!\!\!\!\!\!v(\tau,\tau{-}\omega(x{+}t))\d\tau\!+\!\!\displaystyle\int_{\psi^{-1}(x{+}t)}^{t}\!\!\!\!\!\!\!\!\!\!\!\!\!\!\!\!\!\!\!\!\!v(\tau,x{+}t{-}\tau)\d\tau, & \Omega_3',
			\end{cases}
			\end{equation}
		\end{subequations}
	By the explicit formulas \eqref{A} and \eqref{Hderx} we deduce that for a.e $t\in\left[0,\frac{\ellz}{2}\right]$ the following equalities hold true:
		\begin{subequations}\label{boundaryder}
			\begin{equation}\label{vxpunto}
			v_x(t,0)=-\dot{z}(t)+\dot{v}_0(t)+v_1(t)+\frac{\nu^2}{4}\int_{0}^{t}v(\tau,t{-}\tau)\d\tau,
			\end{equation}
			\begin{equation}\label{vxelle}
			v_x(t,\ell(t))=\frac{1}{1+\elld(t)}\left[\vzd(\ell(t){-}t)-v_1(\ell(t){-}t)-\frac{\nu^2}{4}\int_{0}^{t}v(\tau,\tau+\ell(t){-}t)\d\tau\right].
			\end{equation}
		\end{subequations}
	Of course, by an iteration argument, this shows that the functions $v_x(\cdot,0)$ and $v_x(\cdot,\ell(\cdot))$, and thus $u_x(\cdot,0)$ and $u_x(\cdot,\ell(\cdot))$, are well-defined for almost every time.
	\begin{rmk}
		The function $A$ depends on $\ell$ via the function $\omega$ (see \eqref{phipsidef} and \eqref{omegadef}) and the function $H$ depends on $\ell$ via the set $R$ (see \eqref{rettangoli} and \eqref{bordi}) and depends on $v$ explicitely, so one should write $A_\ell$ and $H_{v,\ell}$. However in the whole paper we shall write only $A$ and $H$ to avoid too heavy notations.
	\end{rmk}
	\begin{rmk}\label{zeroext}
		As already said, in the whole paper the solution $u$ (and hence $v$) and the functions $A$ and $H$ are extended to zero outside $\Omega$.
	\end{rmk}
\subsection{Convergence assumptions on the data}\label{hypotheses}
	Now that we have precised all the notations and properties of solutions of the coupled problem \eqref{problemu}\&\eqref{Griffithcrit} we can state the issue we want to address in this paper. We start listing all the hypotheses on the limit data and on the sequences of data we will assume in the whole paper.\par
	\textbf{The limit data.} Let us fix $\nu\geq 0$, $\ellz>0$, functions $u_0$, $u_1$, $w$ satisfying \eqref{bdryregularity}, and a measurable function $\kappa\colon[\ell_0,+\infty)\to(0,+\infty) $ which belongs to $\widetilde C^{0,1} ([\ellz,+\infty))$ and so in particular it fulfills property \eqref{kappaeps}.\par 
	We extend $u_0$, $u_1$ to the whole $[0,+\infty)$ setting them to be identically zero outside $[0,\ell_0]$ (notice that by compatibility condition $u_0$ belongs to $H^1(0,+\infty)$) and we extend $\kappa$ to $[0,+\infty)$ setting $\kappa(x)=\kappa(\ellz)$ for $x\in[0,\ellz]$.\par 
	\textbf{The sequences of data.} Let us consider a sequence of positive real numbers $\{\ellzk\}_{k\in\enne}$, a sequence of non negative real numbers $\{\nuk\}_{k\in\enne}$, sequences of functions $\{u_0^k\}_{k\in\enne}$, $\{u_1^k\}_{k\in\enne}$ and $\{w^k\}_{k\in\enne}$ satisfying \eqref{bdryregularity} replacing $\ellz$ by $\ellzk$ and a sequence of functions $\{\kappak\}_{k\in\enne}$ such that $\kappak\colon [\ellzk,+\infty)\to(0,+\infty)$ belongs to $\widetilde C^{0,1} ([\ellzk,+\infty))$ for every $k\in\enne$ (and hence it fulfills property \eqref{kappaeps}, replacing $\ellz$ by $\ellzk$).\par 
	As before we extend $u_0^k$, $u_1^k$ to the whole $[0,+\infty)$ setting them to be identically zero outside $[0,\ellzk]$ and we extend $\kappak$ to $[0,+\infty)$ setting $\kappak(x)=\kappak(\ellzk)$ for $x\in[0,\ellzk]$.\par 
	\textbf{The convergence assumptions.} As $k\to +\infty$ we assume:
	\begin{subequations}\label{conv}
		\begin{equation}\label{convlnu}
			\ellzk\to\ellz\quad\mbox{ and }\quad\nuk\to\nu;
		\end{equation}
		\begin{equation} \label{convdata}
			u_0^k\to u_0 \mbox{ in } H^1(0,+\infty),\, u_1^k\to u_1\mbox{ in }L^2(0,+\infty)\mbox{ and } w^k\to w\mbox{ in }\widetilde  H^1(0,+\infty);
		\end{equation}
		\begin{equation}\label{convkappa}
			\kappak\to\kappa\mbox{ in }C^0([0,X]) \mbox{ for every }X>0.
		\end{equation}	
	\end{subequations}
\subsection{The main result}  
Let now $(u,\ell)$ and $(u^k,\ellk)$ be the solutions of the coupled problem given by Theorem~\ref{exuniq} corresponding to the limit data and to the $k$th term of the sequence of data, respectively. The principal result of the paper, stated in Theorem~\ref{finalthm}, affirms that under the assumptions of this first Section the following convergences hold true for every $T>0$:
	\begin{equation*}
	\begin{aligned}
	&\bullet\elldk\to\elld \mbox{ in } L^1(0,T),\mbox{ and thus }\ellk\to\ell\mbox{ uniformly in }[0,T];\\
	&\bullet u^k\to u\mbox{ uniformly in }[0,T]\times[0,+\infty);\\
	&\bullet u^k\to u\mbox{ in }H^1((0,T)\times(0,+\infty));\\
	&\bullet u^k\to u\mbox{ in }C^0([0,T];H^1(0,+\infty))\mbox{ and in } C^1([0,T];L^2(0,+\infty));\\
	&\bullet u^k_x(\cdot,0)\to u_x(\cdot,0)\mbox{ and }\sqrt{1-\elldk(\cdot)^2}u^k_x(\cdot,\ellk(\cdot))\to \sqrt{1-\elld(\cdot)^2}u_x(\cdot,\ell(\cdot))\mbox{ in }L^2(0,T).
	\end{aligned}
	\end{equation*}
	We recall that the term $\sqrt{1-\elld(\cdot)^2}u_x(\cdot,\ell(\cdot))$ is, up to the constant $1/\sqrt{2}$ and up to the sign, the square root of the dynamic energy release rate $G_{\elld(\cdot)}(\cdot)$, see \eqref{Griffregular}.
	\begin{rmk}\label{Remcont}
		If instead of considering the coupled problem, we study system \eqref{problemu} with a prescribed debonding front, then we obtain an analogous continuous dependence result. This analysis will be performed in Section~\ref{sec2}, see \eqref{allconv}, Remark~\ref{auxiliar} and also Propositions~\ref{uniform}, \ref{sobolev}, \ref{Cone} and \ref{boundary}.
	\end{rmk}\noindent
	To prove the Theorem we will exploit the sequence of auxiliary functions $\vk(t,x)=e^{\nuk t/2}u^k(t,x)$, whose boundary and initial data are the functions $\vzk$, $\vuk$ and $\zk$ given by \eqref{datav}. We recall that for $T<\frac \ellz 2$ they can be expressed using representation formula \eqref{repformula} as
	\begin{equation}\label{duhamv}
		\vk(t,x)=\Ak(t,x)+\frac{(\nuk)^2}{8}\Hk(t,x),\quad\quad\mbox{ for every }(t,x)\in[0,T]\times[0,+\infty),
	\end{equation}
	where the function $\Ak$ is as in \eqref{A} with the obvious changes, while $H^k(t,x)\!=\displaystyle\!\!\!\iint_{R^k(t,x)}\!\!\!\!\!\!\!\!\!\!\!\!\!\!\!\vk(\tau,\sigma)\d \sigma\d\tau$. As stressed in Remark~\ref{zeroext} they both are extended to zero outside $\overline{\Omega^k}$.
	\begin{rmk}
		By \eqref{datav} it is easy to see that convergence hypotheses \eqref{convlnu} and \eqref{convdata} yield the same kind of convergence for the functions $\vzk$, $\vuk$ and $\zk$.
	\end{rmk}
	In the next two Sections we analyse the convergence of the pair $(\vk,\ellk)$ instead of the one of the pair $(u^k,\ellk)$. Indeed the transformed pair $(\vk,\ellk)$ is easier than $(u^k,\ellk)$ to handle with, since in \eqref{Gamma} and \eqref{Acca} inside the integral it appears the function itself, and not its time derivative (see \eqref{Gzero}). We are able to prove that the convergences listed just above hold true for the auxiliary function $v^k$, and thus, since it is linked to $u^k$ via the equality $\vk(t,x)=e^{\nuk t/2}u^k(t,x)$, the result is easily transferred to the solution $u^k$ of the coupled problem.
	\begin{rmk}[\textbf{Notation}]
		From now on during all the estimates the symbol $C$ is used to denote a constant, which may change from line to line, which does not depend on $k$. The symbol $\eps^k$ is instead used to denote the $k$th term of a generic infinitesimal sequence.
	\end{rmk}

	\section{A priori convergence of the debonding front}\label{sec2}
	In this Section we prove that if we assume a priori the validity of certain suitable convergence (uniform and in $W^{1,1}$) on the sequence of debonding fronts $\{\ellk\}_{k\in\enne}$ in a time interval $[0,T]$, then the sequence of auxiliary functions $\{\vk\}_{k\in\enne}$ converges to $v$ in the natural spaces. First of all we prove an equiboundedness result for the sequence $\{\vk\}_{k\in\enne}$:
	\begin{prop}
		\label{bdd}
		Assume \eqref{convlnu}, \eqref{convdata} and let us denote by $N$ the maximum value of $\nuk$. If $T< \min\left\{\frac{\ellz}{2},\frac{2}{N^2\ellz}\right\}$, then the functions $\vk$ are uniformly bounded in $C^0([0,T]\times[0,+\infty))$.
	\end{prop}
	\begin{proof}
		We exploit representation formula \eqref{duhamv} and we estimate:
		\begin{align*}
		\Vert \vk\Vert_{C^0([0,T]\times[0,+\infty))}&\leq\Vert \Ak\Vert_{C^0([0,T]\times[0,+\infty))}+\frac{(\nuk)^2}{8}\Vert \Hk\Vert_{C^0([0,T]\times[0,+\infty))}\\
		&\leq\Vert\Ak\Vert_{C^0([0,T]\times[0,+\infty))}+\frac{N^2}{8}|\Omega^k_T|\Vert \vk\Vert_{C^0([0,T]\times[0,+\infty))}\\
		&\leq\Vert\Ak\Vert_{C^0([0,T]\times[0,+\infty))}+\frac{N^2\ellz T}{4}\Vert \vk\Vert_{C^0([0,T]\times[0,+\infty))}.
		\end{align*}
		Since by hypothesis $T\leq \frac{2}{N^2\ellz}$ we deduce that:
		\begin{equation*}
			\Vert \vk\Vert_{C^0([0,T]\times[0,+\infty))}\leq2\Vert \Ak\Vert_{C^0([0,T]\times[0,+\infty))}.
		\end{equation*}
		By the explicit expression of $\Ak$ given by \eqref{A} and using \eqref{convdata} it is easy to get the equiboundedness of $\Ak$ in $C^0([0,T]\times[0,+\infty))$ and so we conclude.
	\end{proof}
Before starting the analysis of the convergence of the sequence $\{\Ak\}_{k\in\enne}$ we state several Lemmas~regarding the convergence of the sequence $\{\omega^k\}_{k\in\enne}$ appearing in formulas \eqref{bordi}, \eqref{A} and \eqref{Hder}.
	\begin{lemma}
		\label{unifinv}
		Let $f^k\colon [a,b]\to\erre$ be a sequence of continuous and invertible functions and assume $f^k$ uniformly converges to a continuous and invertible function $f\colon [a,b]\to\erre$. Then
		$\lim\limits_{k\to +\infty}\max\limits_{y\in D^k_f(a,b)}|(f^k)^{-1}(y)-f^{-1}(y)|= 0$, where $D^k_f(a,b):=f^k([a,b])\cap f([a,b])$.
	\end{lemma}
	\begin{proof}
		For $y\in D^k_f(a,b)$ it holds:
		\begin{equation}
		\label{invf}
			|(f^k)^{-1}(y)-f^{-1}(y)|=|f^{-1}(f((f^k)^{-1}(y)))-f^{-1}(y)|.
		\end{equation}
		Since $f$ is continuous, $f^{-1}$ is uniformly continuous on the compact interval $f([a,b])$ and so by \eqref{invf} to conclude it is enough to prove that $\max\limits_{y\in f^k([a,b])}|f((f^k)^{-1}(y))-y|\to 0$ as $k\to+\infty$. So let us take $y\in f^k([a,b])$ and reason as follows:
	\begin{equation*}
		|f((f^k)^{-1}(y))-y|=|f((f^k)^{-1}(y))-f^k((f^k)^{-1}(y))|\leq \Vert f^k-f\Vert_{C^0{([a,b])}}.
	\end{equation*}
	Since by hypothesis $f^k$ uniformly converges to $f$ in $[a,b]$ the proof is complete.
	\end{proof}
	As we did in Lemma~\ref{unifinv} we now introduce the following notation: given a time $T>0$ we define $D^k_\psi(0,T):=\psik([0,T])\cap\psi([0,T])$ and $D^k_\varphi(0,T):=\varphik([0,T])\cap\varphi([0,T])$. We notice that we can rewrite them as:
	\begin{equation*}
		D^k_\psi(0,T)=[\ellzk\vee\ellz,\psik(T)\wedge\psi(T)]\quad\mbox{ and }\quad D^k_\varphi(0,T)=[-(\ellz\wedge\ellzk),\varphi(T)\wedge\varphi^k(T)].
	\end{equation*}
	\begin{lemma}
		\label{omega}
		If $\ellk$ uniformly converges to $\ell$ in $[0,T]$, then $\lim\limits_{k\to +\infty}\max\limits_{t\in D^k_\psi(0,T)}|\omk(t)-\omega(t)|=0$. If \eqref{convlnu} holds and $\elldk\to\elld$ in $L^1(0,T)$, then $\displaystyle\lim\limits_{k\to +\infty}\int_{D^k_\psi(0,T)}^{}|\omdk(t)-\omd(t)|\d t= 0$.
	\end{lemma}
	\begin{proof}
		Assume that $\ellk\to\ell$ uniformly in $[0,T]$, then obviously $\psik\to\psi$ uniformly in $[0,T]$ and so by Lemma~\ref{unifinv} we get $\lim\limits_{k\to +\infty}\max\limits_{t\in D^k_\psi(0,T)}|(\psik)^{-1}(t)-\psi^{-1}(t)|=0$. Take now $t\in D^k_\psi(0,T)$, then
		\begin{align*}
		|\omk(t)-\omega(t)|&\leq|\varphi^k((\psik)^{-1}(t))-\varphi((\psik)^{-1}(t))|+|\varphi((\psik)^{-1}(t))-\varphi(\psi^{-1}(t))|\\
		&\leq \Vert \ellk-\ell\Vert_{C^0([0,T])}+|(\psik)^{-1}(t)-\psi^{-1}(t)|,
		\end{align*}
		and hence we deduce $\lim\limits_{k\to +\infty}\max\limits_{t\in D^k_\psi(0,T)}|\omk(t)-\omega(t)|=0$.\par 
		Now assume that $\elldk\to\elld$ in $L^1(0,T)$. Notice that by \eqref{convlnu} this implies $\ellk\to\ell$ uniformly in $[0,T]$, and so we have:
		\begingroup
		\allowdisplaybreaks
		\begin{align*}
			\int_{D^k_\psi(0,T)}\!\!\!\!\!\!\!\!\!\!\!\!\!\!|\omdk(t)-\omd(t)|\d t &=\int_{D^k_\psi(0,T)}\left|\frac{1-\elldk((\psik)^{-1}(t))}{1+\elldk((\psik)^{-1}(t))}-\frac{1-\elld(\psi^{-1}(t))}{1+\elld(\psi^{-1}(t))}\right|\d t\\
			&\leq 2\int_{D^k_\psi(0,T)}\left|\elldk((\psik)^{-1}(t))-\elld(\psi^{-1}(t))\right|\d t\\
			&\leq 2\left(\!\int_{D^k_\psi(0,T)}\!\!\!\!\!\!\!\!\!\!\!\!\!\!|\elldk((\psik)^{-1}(t))-\elld((\psik)^{-1}(t))|\d t+\int_{D^k_\psi(0,T)}\!\!\!\!\!\!\!\!\!\!\!\!\!\!|\elld((\psik)^{-1}(t))-\elld(\psi^{-1}(t))|\d t\!\right)\\
			&\leq2\left(2\int_{0}^{T}\left|\elldk(s)-\elld(s)\right|\d s+\int_{D^k_\psi(0,T)}\left|\elld((\psik)^{-1}(t))-\elld(\psi^{-1}(t))\right|\d t\right).
		\end{align*}
		\endgroup
		By assumption the first term in the last line goes to zero as $k\to +\infty$, while for the second term we reason as follows. We fix $\eps>0$ and we consider $f_\eps\in C^0([0,T])$ such that $\Vert\elld-f_\eps\Vert_{L^1(0,T)}\leq \eps$, so we can estimate:
		\begingroup
		\allowdisplaybreaks
		\begin{align*}
			&\quad\,\,\int_{D^k_\psi(0,T)}\left|\elld((\psik)^{-1}(t))-\elld(\psi^{-1}(t))\right|\d t\\ &\leq\int_{D^k_\psi(0,T)}\left|\elld((\psik)^{-1}(t))-f_\eps((\psik)^{-1}(t))\right|\d t+\int_{D^k_\psi(0,T)}\left|f_\eps((\psik)^{-1}(t))-f_\eps(\psi^{-1}(t))\right|\d t\\
			&\quad+\int_{D^k_\psi(0,T)}\left|f_\eps(\psi^{-1}(t))-\elld(\psi^{-1}(t))\right|\d t\\
			&\leq 2\Vert\elld-f_\eps\Vert_{L^1(0,T)}+\int_{D^k_\psi(0,T)}\left|f_\eps((\psik)^{-1}(t))-f_\eps(\psi^{-1}(t))\right|\d t+2\Vert\elld-f_\eps\Vert_{L^1(0,T)}\\
			&\leq 4\eps+\int_{D^k_\psi(0,T)}\left|f_\eps((\psik)^{-1}(t))-f_\eps(\psi^{-1}(t))\right|\d t.
		\end{align*}
		\endgroup
		By dominated convergence the last integral goes to zero as $k\to+\infty$ and so by the arbitrariness of $\eps$ we get the result.
	\end{proof}

	\begin{lemma}
		\label{composition}
		Let $f^k$ be a sequence of $L^2(\erre)$-functions converging to $f$ strongly in $L^2(\erre)$. If \eqref{convlnu} holds and $\elldk\to\elld$ in $L^1(0,T)$, then 
		\begin{equation*}
			\lim\limits_{k\to +\infty}\int_{D^k_\psi(0,T)}^{}|f^k(-\omk(s))\omdk(s)-f(-\omega(s))\omd(s)|^2\d s=0.
		\end{equation*}
	\end{lemma}
	\begin{proof}
		It is enough to estimate:
		\begingroup
		\allowdisplaybreaks
		\begin{align*}
			&\quad\,\int_{D^k_\psi(0,T)}^{}|f^k(-\omk(s))\omdk(s)-f(-\omega(s))\omd(s)|^2\d s\\
			&\leq 2\int_{D^k_\psi(0,T)}^{}\!\!\!\!\!\!\!\!\!\!\!\!\!|f^k(-\omk(s))\omdk(s)-f(-\omk(s))\omdk(s)|^2\d s+2\int_{D^k_\psi(0,T)}^{}\!\!\!\!\!\!\!\!\!\!\!\!\!|f(-\omk(s))\omdk(s)-f(-\omega(s))\omd(s)|^2\d s\\
			&\leq 2\Vert f^k-f\Vert^2_{L^2(\erre)}+2\int_{D^k_\psi(0,T)}^{}|f(-\omk(s))\omdk(s)-f(-\omega(s))\omd(s)|^2\d s.
		\end{align*}
		\endgroup
		Here we used the uniform bound of $\omdk$, see \eqref{omegadot}. By assumption the first term in the last line vanishes as $k\to+\infty$, while for the second integral we reason as in the proof of Lemma~\ref{omega}: for $\eps>0$ fixed let us consider $f_\eps\in C^0_c(\erre)$ satisfying $\Vert f-f_\eps\Vert^2_{L^2(\erre)}\leq\eps$, then we have:
		\begingroup
		\allowdisplaybreaks
		\begin{align*}
			&\quad\,\int_{D^k_\psi(0,T)}^{}|f(-\omk(s))\omdk(s)-f(-\omega(s))\omd(s)|^2\d s\\
			&\leq 3\int_{D^k_\psi(0,T)}^{}\!\!\!\!\!\!\!\!\!\!\!\!\!\!\!\!|f(-\omk(s))\omdk(s)-f_\eps({-}\omk(s))\omdk(s)|^2\d s+3\int_{D^k_\psi(0,T)}^{}\!\!\!\!\!\!\!\!\!\!\!\!\!\!\!\!|f_\eps({-}\omk(s))\omdk(s)-f_\eps(-\omega(s))\omd(s)|^2\d s\\
			&\quad+3\int_{D^k_\psi(0,T)}^{}\!\!\!\!\!\!\!\!\!\!\!\!\!\!\!\!|f_\eps({-}\omega(s))\omd(s)-f(-\omega(s))\omd(s)|^2\d s\\
			&\leq 3\int_{\erre}^{}\!\!|f(x)-f_\eps(x)|^2\d x+3\int_{D^k_\psi(0,T)}^{}\!\!\!\!\!\!\!\!\!\!\!\!\!\!|f_\eps(-\omk(s))\omdk(s)-f_\eps(-\omega(s))\omd(s)|^2\d s+3\int_{\erre}^{}\!\!|f(x)-f_\eps(x)|^2\d x\\
			&\leq 6\eps+3\int_{D^k_\psi(0,T)}^{}|f_\eps(-\omk(s))\omdk(s)-f_\eps(-\omega(s))\omd(s)|^2\d s.
		\end{align*}
		\endgroup
		By dominated convergence the last integral goes to zero as $k\to+\infty$. Indeed exploiting Lemma~\ref{omega} we deduce that, up to subsequences (not relabelled), the function $|\omdk-\omd|\chi_{D^k_\psi(0,T)}$ (here and henceforth $\chi$ denotes the characteristic function of a set) vanishes almost everywhere on a bounded interval (the intervals $D^k_\psi(0,T)$ are all contained for instance in $[0,\psi(T)+1]$). By continuity of $f_\eps$ and since by assumptions $\chi_{D^k_\psi(0,T)}\to\chi_{[\ellz,\psi(T)]}$ almost everywhere as $k\to +\infty$, this implies that also $|f_\eps(-\omk)\omdk-f_\eps(-\omega)\omd|^2 \chi_{D^k_\psi(0,T)}$ vanishes almost everywhere on that bounded interval. Since the limit does not depend on the subsequence we conclude.\par 
		Thus by the arbitrariness of $\eps$ we get the result.
	\end{proof}
	Now that we have established some convergence results of the sequence $\{\omega^k\}_{k\in\enne}$ we can start to study how the sequence $\{\Ak\}_{k\in\enne}$ behaves under different convergence assumptions on $\{\ellk\}_{k\in\enne}$.
	\begin{prop}
		\label{uniformA}
	Assume \eqref{convdata} and let $T<\frac{\ellz}{2}$. If $\ellk$ uniformly converges to $\ell$ in $[0,T]$, then $\Ak$ uniformly converges to $ A$ in $[0,T]\times[0,+\infty)$.
	\end{prop}
	\begin{proof}
		We assume without loss of generality that $\ellz<\ellzk$, the other cases being analogous. As in the whole paper we exploit explicit formula \eqref{A}, so we need to deal with some different cases separately. We thus consider the following partition of $[0,T]\times[0,+\infty)$, see Figure~\ref{FigLambdas}:
			\begin{alignat}{4}\label{partition}
			&\Lambda^k_1:=(\Omega'_1)_T,\quad\quad&&\Lambda^k_2:=(\Omega'_2)_T,\quad\quad &\Lambda^k_3:=(\Omega'^k_1)_T\cap(\Omega'_3)_T,\nonumber\\
			&\Lambda^k_4:=(\Omega'^k_1)_T\diff\Omega_T,\quad\quad&&\Lambda^k_5:=(\Omega'^k_3)_T\cap(\Omega'_3)_T,\quad\quad&\Lambda^k_6:=(\Omega'^k_3)_T\diff\Omega_T,\\
			&\Lambda^k_7:=(\Omega'_3)_T\diff\Omega^k_T,\quad\quad &&\Lambda^k_8:=\big([0,T]\times[0,+\infty)\big)\setminus\bigcup_{i=1}^{7}\Lambda^k_i.\nonumber
			\end{alignat}
		If $(t,x)\in \Lambda^k_1$, then
		\begin{align*}
		|\Ak(t,x)-A(t,x)|\leq\Vert\vzk-\vz\Vert_{C^0([0,+\infty))}+\frac{\sqrt{\ellz}}{2}\Vert\vuk-\vu\Vert_{L^2(0,+\infty)}.
		\end{align*}
		If $(t,x)\in \Lambda^k_2$, then 
		\begin{align*}
		|\Ak(t,x)-A(t,x)|\leq \Vert\zk-z\Vert_{C^0([0,T])}+ \Vert\vzk-\vz\Vert_{C^0([0,+\infty))}+\frac{\sqrt{\ellz}}{2}\Vert\vuk-\vu\Vert_{L^2(0,+\infty)}.
		\end{align*}
		If $(t,x)\in \Lambda^k_3$, we first notice that $\vz(x{+}t)=0$ and that $-\omega(\ellzk)\leq-\omega(x{+}t)\leq \ellz\leq x{+}t\leq\ellzk$, then we estimate:
		\begingroup
		\allowdisplaybreaks
		\begin{align*}
			&\quad\,|\Ak(t,x)-A(t,x)|\\
			&\leq\frac 12 |\vzk(x{-}t)-\vz(x{-}t)|+\frac 12|\vzk(x{+}t)+\vz(-\omega(x{+}t))|+\frac 12\left|\int_{x{-}t}^{x{+}t}\vuk(s)\d s-\int_{x{-}t}^{-\omega(x{+}t)}\!\!\!\!\!\!\!\vu (s)\d s\right|\\
			&\leq\Vert\vzk-\vz\Vert_{C^0([0,+\infty))}+\frac{\sqrt{\ellz}+\sqrt{\ellzk}}{2}\Vert\vuk-\vu\Vert_{L^2(0,+\infty)}+\frac 12|\vz(-\omega(x{+}t))|+\frac 12 \left|\int_{-\omega(x{+}t)}^{x{+}t}\!\!\!\!\!\!\!\vu(s)\d s\right|\\
			&\leq \Vert\vzk-\vz\Vert_{C^0([0,+\infty))}+C\Vert\vuk-\vu\Vert_{L^2(0,+\infty)}+\int_{-\omega(\ellzk)}^{\ellz}(|\vzd(s)|+|\vu(s)|)\d s.
		\end{align*}
		\endgroup
		If $(t,x)\in \Lambda^k_4$, we notice that $-\omega(\ellzk)\le x{-}t\le x{+}t\le \ellzk$ and hence we get:
		\begin{align*}
		|\Ak(t,x)-A(t,x)|&=|\Ak(t,x)|\leq\int_{-\omega(\ellzk)}^{\ellzk}|\vzdk(s)|\d s+\frac 12\int_{-\omega(\ellzk)}^{\ellzk}|\vuk(s)|\d s\\
		&\leq C\Vert\vzdk-\vzd\Vert_{L^2(0,+\infty)}+C\Vert\vuk-\vu\Vert_{L^2(0,+\infty)}+\int_{-\omega(\ellzk)}^{\ellz}\!\!\!\!\!\!\!\left(|\vzd(s)|+|\vu(s)|\right)\d s.
		\end{align*}
		If $(t,x)\in \Lambda^k_5$, then
		\begingroup
		\allowdisplaybreaks
		\begin{align*}
			&\quad\,|\Ak(t,x)-A(t,x)|\\
			&\leq\frac 12 \Vert\vzk-\vz\Vert_{C^0([0,+\infty))}+\frac 12 |\vzk(-\omk(x{+}t))-\vz(-\omega(x{+}t))|+\frac 12 \left|\int_{x{-}t}^{-\omk(x{+}t)}\!\!\!\!\!\!\!\!\!\!\!\!\!\!\!\!\!\!\!\!\vuk(s)\d s-\int_{x{-}t}^{-\omega(x{+}t)}\!\!\!\!\!\!\!\!\!\!\!\!\!\!\!\!\!\!\!\!\vu(s)\d s\right|\\
			&\leq \Vert\vzk-\vz\Vert_{C^0([0,+\infty))}+\frac 12 |\vz(-\omk(x{+}t))-\vz(-\omega(x{+}t))|+C\Vert\vuk-\vu\Vert_{L^2(0,+\infty)}\\
			&\quad+\frac 12 \left|\int_{-\omega(x{+}t)}^{-\omk(x{+}t)}\!\!\!\!|\vu(s)|\d s\right|\\
			&\leq \Vert\vzk-\vz\Vert_{C^0([0,+\infty))}+\max\limits_{r\in D^k_\psi(0,T)}|\vz(-\omk(r))-\vz(-\omega(r))|+C\Vert\vuk-\vu\Vert_{L^2(0,+\infty)}\\
			&\quad+\max\limits_{r\in D^k_\psi(0,T)}\left|\int_{-\omega(r)}^{-\omk(r)}|\vu(s)|\d s\right|.
		\end{align*}
		\endgroup
		If $(t,x)\in \Lambda^k_6$, we notice that $-\omega(x{+}t)\le x{-}t\le -\omk(x{+}t)$ and hence we get:
		\begingroup
		\allowdisplaybreaks
		\begin{align*}
			|\Ak(t,x)-A(t,x)|&=|\Ak(t,x)|\leq\frac 12 |\vzk(x{-}t)-\vzk(-\omk(x{+}t))|+\frac 12 \int_{x{-}t}^{-\omk(x{+}t)}|\vuk(s)|\d s\\
			&\leq\frac 12\int_{x{-}t}^{-\omk(x{+}t)}(|\vzdk(s)|+|\vuk(s)|)\d s\\
			&\leq C\Vert\vzdk-\vzd\Vert_{L^2(0,+\infty)}+C\Vert\vuk-\vu\Vert_{L^2(0,+\infty)}+\int_{x{-}t}^{-\omk(x{+}t)}\!\!\!\!\!\!\!\left(|\vzd(s)|+|\vu(s)|\right)\d s\\
			&\leq C\Vert\vzdk-\vzd\Vert_{L^2(0,+\infty)}\!+\!C\Vert\vuk-\vu\Vert_{L^2(0,+\infty)}\!\!+\!\!\!\!\max\limits_{r\in D^k_\psi(0,T)}\int_{-\omega(r)}^{-\omk(r)}\!\!\!\!\!\!\!\!\!\!\!\!\!\!(|\vzd(s)|+|\vu(s)|)\d s.
		\end{align*}
		\endgroup
		If $(t,x)\in \Lambda^k_7$ one reasons just as above, while if $(t,x)\in \Lambda^k_8$ there is nothing to prove since $A^k(t,x)=A(t,x)=0$.\par
		We conclude exploiting Lemma~\ref{omega} and using \eqref{convdata}.			
	\end{proof}

	\begin{prop}
		\label{H1A}
		Assume \eqref{convlnu}, \eqref{convdata} and let $T<\frac{\ellz}{2}$. If $\elldk\to\elld$ in $L^1(0,T)$, then $\Ak\to A$ in $H^1((0,T)\times(0,+\infty))$.
	\end{prop}
	\begin{proof}
		\begin{figure}
			\centering
			\includegraphics[scale=.8]{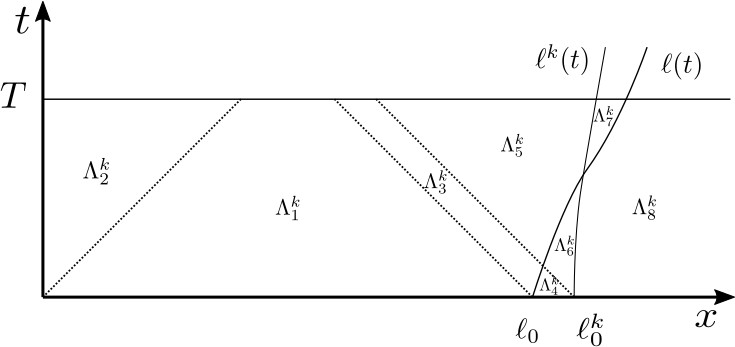}
			\caption{The partition of the set $[0,T]\times[0,+\infty)$ via the sets $\Lambda^k_i$, for $i=1,\dots,8$, in the case $\ellz<\ellzk$.}\label{FigLambdas}
		\end{figure}
		First of all we notice that our hypothesis imply $\ellk$ uniformly converges to $\ell$ in $[0,T]$ and hence by Proposition~\ref{uniformA} we deduce that $\Ak\to A$ in $L^2((0,T)\times(0,+\infty))$, so we only have to prove that the same kind of convergence holds true for $\Akt$ and $\Akx$. We assume without loss of generality that $\ellz<\ellzk$, the other cases being analogous. We then consider again the partition \eqref{partition} used in the proof of previous Proposition, see also Figure \ref{FigLambdas}. So we have:
		\begin{equation*}
			\Vert \Akt-A_t\Vert^2_{L^2((0,T)\times(0,+\infty))}=\sum_{i=1}^{7}\iint_{\Lambda_i^k}|\Akt(t,x)-A_t(t,x)|^2\d x\d t.
		\end{equation*}
		By \eqref{convdata} the integrals over $\Lambda^k_1$ and $\Lambda^k_2$ goes to zero as $k\to+\infty$. For the others we start to estimate from $\Lambda_3^k$:
		\begingroup
		\allowdisplaybreaks
		\begin{align*}
			&\quad\,\iint_{\Lambda_3^k}|\Akt(t,x)-A_t(t,x)|^2\d x\d t\\
			&\leq C\iint_{\Lambda_3^k}(|\vzdk(x{-}t)-\vzd(x{-}t)|^2+|\vuk(x{-}t)-\vu(x{-}t)|^2)\d x\d t\\
			&\quad\,+ C\iint_{\Lambda_3^k}\left(\left|\vzdk(x{+}t)-\vzd(-\omega(x{+}t))\omd(x{+}t)\right|^2+\left|\vuk(x{+}t)+\vu(-\omega(x{+}t))\omd(x{+}t)\right|^2\right)\d x\d t\\
			&\leq C\left(\Vert\vzdk-\vzd\Vert^2_{L^2(0,+\infty)}\!+\!\Vert\vuk-\vu\Vert^2_{L^2(0,+\infty)}\!+\!\!\iint_{\Lambda_3^k}\!\!\!\Big((|\vzd|^2+|\vu|^2)({-}\omega(x{+}t))\Big)\omd(x{+}t)^2\d x\d t\right)\\
			&\leq C\left(\Vert\vzdk-\vzd\Vert^2_{L^2(0,+\infty)}+\Vert\vuk-\vu\Vert^2_{L^2(0,+\infty)}+\int_{-\omega(\ellzk)}^{\ellz}(|\vzd(s)|^2+|\vu(s)|^2)\d s\right).
		\end{align*}
		\endgroup
		As regards $\Lambda^k_4$ we have:
		\begingroup
		\allowdisplaybreaks
		\begin{align*}
			&\quad\,\iint_{\Lambda_4^k}|\Akt(t,x)-A_t(t,x)|^2\d x\d t=\iint_{\Lambda_4^k}|\Akt(t,x)|^2\d x\d t\\
			&\leq C\left(\iint_{\Lambda_4^k}(|\vzdk(x{-}t)|^2+|\vuk(x{-}t)|^2)\d x\d t+\iint_{\Lambda_4^k}(|\vzdk(x{+}t)|^2+|\vuk(x{+}t)|^2)\d x\d t\right)\\
			&\leq C\left(\int_{-\omega(\ellzk)}^{\ellzk}(|\vzdk(s)|^2+|\vuk(s)|^2)\d s+\int_{\ellz}^{\ellzk}(|\vzdk(s)|^2+|\vuk(s)|^2)\d s\right)\\
			&\leq C\left(\Vert\vzdk-\vzd\Vert^2_{L^2(0,+\infty)}+\Vert\vuk-\vu\Vert^2_{L^2(0,+\infty)}+\int_{-\omega(\ellzk)}^{\ellz}(|\vzd(s)|^2+|\vu(s)|^2)\d s\right).
		\end{align*}
		\endgroup
		We then consider $\Lambda^k_6\cup\Lambda^k_7$, so that:
		\begin{align*}
			\iint_{\Lambda_6^k\cup\Lambda^k_7}|\Akt(t,x)-A_t(t,x)|^2\d x\d t=\iint_{\Lambda_6^k}|\Akt(t,x)|^2\d x\d t+\iint_{\Lambda_7^k}|A_t(t,x)|^2\d x\d t.
		\end{align*}
		Since by assumptions $\ellk\to\ell$ uniformly in $[0,T]$, we deduce $\Lambda_7^k\to\emptyset$ in measure, and so the second integral goes to zero as $k\to+\infty$, while for the first one we estimate:
		\begingroup
		\allowdisplaybreaks
		\begin{align*}
			&\quad\,\iint_{\Lambda_6^k}|\Akt(t,x)|^2\d x\d t\\
			&\leq C\left(\iint_{\Lambda_6^k}\!\!(|\vzdk(x{-}t)|^2+|\vuk(x{-}t)|^2)\d x\d t+\iint_{\Lambda_6^k}\!\!\Big((|\vzdk|^2+|\vuk|^2)(-\omk(x{+}t))\Big)|\omdk(x{+}t)|^2\d x\d t\right)\\
			&\leq C\max\limits_{r\in D^k_\varphi(0,T)}|(\varphi^k)^{-1}(r)-\varphi^{-1}(r)|\int_{0}^{\ellzk}(|\vzdk(s)|^2+|\vuk(s)|^2)\d s\\
			&\quad+C\max\limits_{r\in D^k_\psi(0,T)}|\omk(r)-\omega(r)|\int_{-\omk(\psi(T)\wedge\psik(T))}^{\ellzk}(|\vzdk(s)|^2+|\vuk(s)|^2)\d s\\
			&\leq C\left(\max\limits_{r\in D^k_\varphi(0,T)}|(\varphi^k)^{-1}(r)-\varphi^{-1}(r)|+\max\limits_{r\in D^k_\psi(0,T)}|\omk(r)-\omega(r)|\right)(\Vert\vzdk\Vert^2_{L^2(0,+\infty)}+\Vert\vuk\Vert^2_{L^2(0,+\infty)})\\
			&\leq C\left(\max\limits_{r\in D^k_\varphi(0,T)}|(\varphi^k)^{-1}(r)-\varphi^{-1}(r)|+\max\limits_{r\in D^k_\psi(0,T)}|\omk(r)-\omega(r)|\right).
		\end{align*}
		\endgroup
		Applying Lemma~\ref{unifinv} for the sequence of functions $\{\varphik\}_{k\in\enne}$ and Lemma~\ref{omega} we deduce that this last integral vanishes as $k\to+\infty$. The last term to treat is the integral over $\Lambda^k_5$:
		\begingroup
		\allowdisplaybreaks
		\begin{align*}
			&\quad\,\iint_{\Lambda_5^k}|\Akt(t,x)-A_t(t,x)|^2\d x\d t\\
			&\leq C\iint_{\Lambda_5^k}|\vzdk(x{-}t)-\vzd(x{-}t)|^2\d x\d t+C\iint_{\Lambda_5^k}|\vuk(x{-}t)-\vu(x{-}t)|^2\d x\d t\\
			&\quad+C\iint_{\Lambda_5^k}\left|\left((\vzdk-\vuk)(-\omk(x{+}t))\right)\omdk(x{+}t)-\Big((\vzd-\vu)(-\omega(x{+}t)\Big)\omd(x{+}t)\right|^2\d x\d t\\
			&\leq C\Vert\vzdk-\vzd\Vert^2_{L^2(0,+\infty)}+C\Vert\vuk-\vu\Vert^2_{L^2(0,+\infty)}\\
			&\quad+C\int_{D^k_\psi(0,T)}^{}\left|\left((\vzdk-\vuk)(-\omk(s))\right)\omdk(s)-\Big((\vzd-\vu)(-\omega(s))\Big)\omd(s)\right|^2\d s.
		\end{align*}
		\endgroup
		Applying Lemma~\ref{composition} to this last integral and putting together all the previous estimates, by \eqref{convlnu} and \eqref{convdata} we finally conclude that $\Akt\to A_t$ in $L^2((0,T)\times(0,+\infty))$. Reasoning exactly in the same way one also gets $\Akx\to A_x$ in $L^2((0,T)\times(0,+\infty))$ and so the Proposition~is proved.
	\end{proof}
	Now we can deal with the convergence of the sequence of auxiliary functions $\{\vk\}_{k\in\enne}$. We only need a short Lemma. Before the statement we introduce the following notation: here and henceforth by $A\Delta B$ we mean the symmetric difference of the sets $A$ and $B$; if moreover both sets depend on time and space, we write $(A\Delta B)(t,x)$ instead of $A(t,x)\Delta B(t,x)$.
	\begin{lemma}
		\label{diffconv}
		Let $T<\frac{\ellz}{2}$ and assume $\ellk$ uniformly converges to $\ell$ in $[0,T]$, then the map $(t,x)\mapsto|(R^k\Delta R) (t,x)|$ uniformly converges to zero in $[0,T]\times[0,+\infty)$.
	\end{lemma}
	\begin{proof}
		We assume without loss of generality that $\ellz<\ellzk$, the other cases being analogous. We then consider again the partition of $[0,T]\times[0,+\infty)$ given by the sets $\Lambda^k_i$, for $i=1,\dots,8$, introduced in the proof of Proposition~\ref{uniformA}.\\
		If $(t,x)\in \Lambda^k_1\cup\Lambda^k_2$, then $(R^k\Delta R)(t,x)=\emptyset$ and so $|(R^k\Delta R) (t,x)|=0$.\\
		If $(t,x)\in \Lambda^k_3\cup\Lambda^k_4$, then $(R^k\Delta R)(t,x)\subseteq[0,\psi^{-1}(\ellzk)]\times[-\omega(\ellzk),\ellzk]$ and so 
		\begin{equation*}
			|(R^k\Delta R) (t,x)|\leq \psi^{-1}(\ellzk)(\ellzk+\omega(\ellzk)).
		\end{equation*}
		If finally $(t,x)\in \Lambda^k_5\cup\Lambda^k_6\cup\Lambda^k_7$, then
		\begin{equation*}
		|(R^k\Delta R) (t,x)|\leq T\max \limits_{r\in D^k_\psi(0,T)}|\omk(r)-\omega(r)|.
		\end{equation*}
		We conclude recalling that $\omega(\ellz)=-\ellz$ and exploiting Lemma~\ref{omega}.
	\end{proof}
	\begin{prop}
		\label{uniform}
		 Assume \eqref{convlnu}, \eqref{convdata} and let $T$ be as in Proposition~\ref{bdd}. If $\ellk$ uniformly converges to $\ell$ in $[0,T]$, then $\vk$ uniformly converges to $ v$ in $[0,T]\times[0,+\infty)$.
	\end{prop}
	\begin{proof}
		Exploiting representation formula \eqref{duhamv} we deduce that:
		\begingroup
		\allowdisplaybreaks 
		\begin{align*}
			&\quad\,\,\Vert \vk-v\Vert_{C^0([0,T]\times[0,+\infty))}\\
			&\leq \Vert \Ak-A\Vert_{C^0([0,T]\times[0,+\infty))}+\frac{|(\nuk)^2-\nu^2|}{8}\Vert H\Vert_{C^0([0,T]\times[0,+\infty))}+\frac{(\nuk)^2}{8}\Vert \Hk-H\Vert_{C^0([0,T]\times[0,+\infty))}\\
			&\leq \Vert \Ak-A\Vert_{C^0([0,T]\times[0,+\infty))}+\frac{|(\nuk)^2-\nu^2|}{8}\Vert H\Vert_{C^0([0,T]\times[0,+\infty))}\\
			&\quad+\frac{N^2}{8}\left\Vert \iint_{R^k}|\vk-v|+\iint_{R^k\Delta R}|v|\right\Vert_{C^0([0,T]\times[0,+\infty))}\\
			&\leq \Vert \Ak-A\Vert_{C^0([0,T]\times[0,+\infty))}+\frac{|(\nuk)^2-\nu^2|}{8}\Vert H\Vert_{C^0([0,T]\times[0,+\infty))}+\frac{N^2}{8}|\Omega^k_T|\Vert \vk-v\Vert_{C^0([0,T]\times[0,+\infty))}\\
			&\quad+\frac{N^2}{8}\Vert|R^k\Delta R|\Vert_{C^0([0,T]\times[0,+\infty))}\Vert v\Vert_{C^0([0,T]\times[0,+\infty))}\\
			&\leq \Vert \Ak-A\Vert_{C^0([0,T]\times[0,+\infty))}+\frac{|(\nuk)^2-\nu^2|}{8}\Vert H\Vert_{C^0([0,T]\times[0,+\infty))}+\frac 12 \Vert \vk-v\Vert_{C^0([0,T]\times[0,+\infty))}\\
			&\quad+\frac{N^2}{8}\Vert|R^k\Delta R|\Vert_{C^0([0,T]\times[0,+\infty))}\Vert v\Vert_{C^0([0,T]\times[0,+\infty))},
		\end{align*}
		\endgroup
		and so we get:
		\begin{align*}
			\quad\,\Vert \vk-v\Vert_{C^0([0,T]\times[0,+\infty))}&\leq 2\Vert \Ak-A\Vert_{C^0([0,T]\times[0,+\infty))}+\frac{|(\nuk)^2-\nu^2|}{4}\Vert H\Vert_{C^0([0,T]\times[0,+\infty))}\\
			&\quad+\frac{N^2}{4}\Vert|R^k\Delta R|\Vert_{C^0([0,T]\times[0,+\infty))}\Vert v\Vert_{C^0([0,T]\times[0,+\infty))}.
		\end{align*}
		Letting $k\to+\infty$ we deduce that by Proposition~\ref{uniformA} the first term goes to zero, by \eqref{convlnu} the second one goes trivially to zero and by Lemma~\ref{diffconv} the third one goes to zero too. So we conclude.
	\end{proof}
	\begin{prop}\label{sobolev}
		Assume \eqref{convlnu}, \eqref{convdata} and let $T$ be as in Proposition~\ref{bdd}. If $\elldk\to\elld$ in $L^1(0,T)$, then $\vk\to v$ in $H^1((0,T)\times(0,+\infty))$.
	\end{prop}
\begin{proof}
	First of all we notice that our hypothesis imply $\ellk\to\ell$ uniformly in $[0,T]$ and hence by Proposition~\ref{uniform} we get $\vk\to v$ uniformly in $[0,T]\times[0,+\infty)$ and so in particular in $L^2((0,T)\times(0,+\infty))$. To get the same result for the sequence of time derivatives $\{\vkt\}_{k\in\enne}$ we estimate:
	\begin{align*}
		&\quad\,\,\Vert\vkt-v_t\Vert_{L^2((0,T)\times(0,+\infty))}\\
		&\leq\Vert\Akt-A_t\Vert_{L^2((0,T)\times(0,+\infty))}+\frac{|(\nuk)^2-\nu^2|}{8}\Vert H_t\Vert_{L^2((0,T)\times(0,+\infty))}\!+\!\frac{N^2}{8}\Vert\Hkt-H_t\Vert_{L^2((0,T)\times(0,+\infty))}.
	\end{align*}
	By Proposition~\ref{H1A} we deduce that the first term goes to zero as $k\to+\infty$, by \eqref{convlnu} the second term goes trivially to zero, while for the third one one gets the same result exploiting the explicit formulas for $\Hkt$ and $H_t$ given by \eqref{Hdert}, the fact that $\vk\to v$ uniformly in $[0,T]\times[0,+\infty)$, and reasoning as in the proof of Proposition~\ref{H1A}.\par 
	With the same argument one can show that also $\vkx\to v_x$ in $L^2((0,T)\times(0,+\infty))$ and so the result is proved.
\end{proof}
	\begin{prop}\label{Cone}
		Assume \eqref{convlnu}, \eqref{convdata} and let $T$ be as in Proposition~\ref{bdd}. If $\elldk\to\elld$ in $L^1(0,T)$, then $\vk\to v$ in $C^0([0,T];H^1(0,+\infty))$ and in $C^1([0,T];L^2(0,+\infty))$.
	\end{prop}
	\begin{proof}
		By Proposition~\ref{uniform} we know that $\vk\to v$ uniformly in $[0,T]\times[0,+\infty)$, so to conclude it is enough to prove that
		\begin{equation*}
			\lim\limits_{k\to +\infty}\max\limits_{t\in[0,T]}\Vert\vkt(t)-v_t(t)\Vert_{L^2(0,+\infty)}= 0\quad\quad\mbox{and}\quad\quad\lim\limits_{k\to +\infty}\max\limits_{t\in[0,T]}\Vert\vkx(t)-v_x(t)\Vert_{L^2(0,+\infty)}= 0.
		\end{equation*}
		We actually prove only the validity of the first limit, the other one being analogous. So we fix $t\in[0,T]$ and we assume that $\ell(t)<\ellk(t)$, being the other cases even easier to deal with, then we estimate:
		\begingroup
		\allowdisplaybreaks
		\begin{align}\label{estvt}
			\nonumber\Vert\vkt(t)-v_t(t)\Vert_{L^2(0,+\infty)}&=\int_{0}^{\ell(t)}|\vkt(t,x)-v_t(t,x)|^2\d x+\int_{\ell(t)}^{\ellk(t)}|\vkt(t,x)|^2\d x\\
			&\leq 2\int_{0}^{\ell(t)}|\Akt(t,x)-A_t(t,x)|^2\d x+2\int_{\ell(t)}^{\ellk(t)}|\Akt(t,x)|^2\d x\\
			\nonumber&\quad+2\int_{0}^{\ell(t)}|\Hkt(t,x)-H_t(t,x)|^2\d x+2\int_{\ell(t)}^{\ellk(t)}|\Hkt(t,x)|^2\d x.
		\end{align}
		\endgroup
		Exploiting the explicit formulas \eqref{Hdert} and Proposition~\ref{bdd} it is easy to see that the second term in the last line is bounded by $C\Vert\ellk-\ell\Vert_{C^0([0,T])}$; always by \eqref{Hdert} we deduce that also the first term in the last line goes uniformly to zero in $[0,T]$. We want to remark that the only difficult part to estimate is the following:
		\begingroup
		\allowdisplaybreaks
		\begin{align*}
			&\quad\,\int_{\ellzk-t}^{\ell(t)}\left|\omdk(x{+}t)\int_{0}^{(\psik)^{-1}(x{+}t)}\!\!\!\!\!\vk(\tau,\tau-\omk(x{+}t))\d\tau-\omd(x{+}t)\int_{0}^{\psi^{-1}(x{+}t)}\!\!\!\!\!v(\tau,\tau-\omega(x{+}t))\d\tau\right|^2\d x\\
			&=\int_{\ellzk-t}^{\ell(t)}\left|\omdk(x{+}t)\int_{0}^{T}\vk(\tau,\tau-\omk(x{+}t))\d\tau-\omd(x{+}t)\int_{0}^{T}v(\tau,\tau-\omega(x{+}t))\d\tau\right|^2\d x\\
			&=\int_{\ellzk}^{\psi(t)}\left|\omdk(s)\int_{0}^{T}\vk(\tau,\tau-\omk(s))\d\tau-\omd(s)\int_{0}^{T}v(\tau,\tau-\omega(s))\d\tau\right|^2\d s,
		\end{align*}
		\endgroup
		which goes uniformly to zero applying Lemma~\ref{composition} and recalling that $\vk\to v$ uniformly in $[0,T]\times[0,+\infty)$.\par 
		The first term in the second line in \eqref{estvt} is estimated just as above using hypothesis \eqref{convdata}, while for the second term we reason as follows:
		\begingroup
		\allowdisplaybreaks
		\begin{align*}
			\int_{\ell(t)}^{\ellk(t)}|\Akt(t,x)|^2\d x&\leq 2\int_{\ell(t)}^{\ellk(t)}\!\!\!\!\!\!\!|(\vzdk+\vuk)(x{-}t)|^2\d x+2\int_{\ell(t)}^{\ellk(t)}\!\!\!\!\!\!\!\big|\big((\vzdk+\vuk)(-\omk(x{+}t))\big)\omdk(x{+}t)\big|^2\d x\\
			&\leq 2\int_{-\varphi(t)}^{-\varphik(t)}|(\vzdk+\vuk)(s)|^2\d s+2\int_{-\varphik(t)}^{-\omk(\psi(t))}|(\vzdk+\vuk)(s)|^2\d s\\
			&\leq 2\Vert\vzdk+\vuk-\vzd-\vu\Vert^2_{L^2(0,+\infty)}+2\int_{-\varphi(t)}^{-\omk(\psi(t))}|(\vzd+\vu)(s)|^2\d s,
		\end{align*}
		\endgroup
		which goes uniformly to zero since $-\omk\circ\psi\to -\varphi$ uniformly.\par 
		So we have proved that $\lim\limits_{k\to +\infty}\max\limits_{t\in[0,T]}\Vert\vkt(t)-v_t(t)\Vert_{L^2(0,+\infty)}= 0$ and we conclude.
	\end{proof}
	\begin{prop}\label{boundary}
	Assume \eqref{convlnu}, \eqref{convdata} and let $T$ be as in Proposition~\ref{bdd}. If $\elldk\to\elld$ in $L^1(0,T)$, then $\vkx(\cdot,0)\to v_x(\cdot,0)$ and $\sqrt{1-\elldk(\cdot)^2}\vkx(\cdot,\ellk(\cdot))\to \sqrt{1-\elld(\cdot)^2}v_x(\cdot,\ell(\cdot))$ in $L^2(0,T)$.
	\end{prop}
	\begin{proof}
	 By \eqref{vxpunto} we recall that for a.e. $t\in(0,T)$ the following equality holds true:
		\begin{equation*}
			\vkx(t,0)=-\zdk(t)+\vzdk(t)+\vuk(t)+\frac{(\nuk)^2}{4}\int_{0}^{t}\vk(\tau,t{-}\tau)\d\tau,
		\end{equation*}
		and so using \eqref{convdata} and Propositions~\ref{bdd} and \ref{uniform} it is easy to deduce $\vkx(\cdot,0)\to v_x(\cdot,0)$ in $L^2(0,T)$.\\
		Moreover by \eqref{vxelle} we know that for a.e. $t\in(0,T)$ it holds:
		\begin{align*}
		\vkx(t,\ellk(t))&=\frac{1}{1+\elldk(t)}\left[\vzdk(\ellk(t){-}t)-\vuk(\ellk(t){-}t)-\frac{(\nuk)^2}{4}\int_{0}^{t}\vk(\tau,\tau+\ellk(t){-}t)\d\tau\right]\\
		&=\frac{1}{1+\elldk(t)}\left[\vzdk(\ellk(t){-}t)-\vuk(\ellk(t){-}t)-\frac{(\nuk)^2}{4}\int_{0}^{T}\vk(\tau,\tau+\ellk(t){-}t)\d\tau\right].
		\end{align*}
		We denote by $g^k(t{-}\ellk(t))$ the expression within the square brackets, i.e. $g^k(t{-}\ellk(t))=(1+\elldk(t))\vkx(t,\ellk(t))$, and we estimate:
		\begingroup
		\allowdisplaybreaks
		\begin{align*}
			&\quad\,\int_{0}^{T}\left|\sqrt{1-\elldk(t)^2}\vkx(t,\ellk(t))- \sqrt{1-\elld(t)^2}v_x(t,\ell(t))\right|^2\d t\\
			&=\int_{0}^{T}\left|\frac{\sqrt{1-\elldk(t)^2}}{1+\elldk(t)}g^k(t{-}\ellk(t))-\frac{\sqrt{1-\elld(t)^2}}{1+\elld(t)}g(t{-}\ell(t))\right|^2\d t\\
			&\leq 2\int_{0}^{T}\left|\frac{1}{1+\elldk(t)}\left(\sqrt{1-\elldk(t)^2}g^k(t{-}\ellk(t))-\sqrt{1-\elld(t)^2}g(t{-}\ell(t))\right)\right|^2\d t\\
			&\quad+2\int_{0}^{T}\left|\frac{1}{1+\elldk(t)}-\frac{1}{1+\elld(t)}\right|^2(1-\elld(t)^2)g(t{-}\ell(t))^2\d t\\
			&\leq 2\int_{0}^{T}\left|\sqrt{1-\elldk(t)^2}g^k(t{-}\ellk(t))-\sqrt{1-\elld(t)^2}g(t{-}\ell(t))\right|^2\d t\\
			&\quad+2\int_{0}^{T}\left|\elldk(t)-\elld(t)\right|(1-\elld(t)^2)g(t{-}\ell(t))^2\d t.
		\end{align*}
		\endgroup
		By dominated convergence the last integral vanishes when $k\to+\infty$, so we conclude if we prove that $\sqrt{1-\elldk(\cdot)^2}g^k(\cdot-\ellk(\cdot))\to\sqrt{1-\elld(\cdot)^2}g(\cdot-\ell(\cdot))$ in $L^2(0,T)$. To this aim we continue to estimate:
		\begin{align*}
		&\quad\,\int_{0}^{T}\left|\sqrt{1-\elldk(t)^2}g^k(t{-}\ellk(t))- \sqrt{1-\elld(t)^2}g(t{-}\ell(t))\right|^2\d t\\
		&\leq 2\int_{0}^{T}(1{-}\elldk(t)^2)\left|\big(g^k{-}g\big)(t{-}\ellk(t))\right|^2\d t+2\int_{0}^{T}\left|\sqrt{1{-}\elldk(t)^2}g(t{-}\ellk(t)){-}\sqrt{1{-}\elld(t)^2}g(t{-}\ell(t))\right|^2\!\!\!\!\d t.
		\end{align*}
		By \eqref{convlnu}, \eqref{convdata} and exploiting Proposition~\ref{uniform} it is easy to see that $g^k(\cdot)\to g(\cdot)$ in $L^2(-\infty,0)$ and so reasoning as in the proof of Lemma~\ref{composition} we get both terms go to zero as $k\to+\infty$. Hence we conclude.
	\end{proof}
	Summarising, in this Section we have obtained the following result: if we assume \eqref{convlnu}, \eqref{convdata} and if for some $T<\min\left\{\frac{\ellz}{2},\frac{2}{N^2\ellz}\right\}$ we know that $\dot{\ellk}\to \dot\ell$ in $L^1(0,T)$ (and hence $\ellk$ uniformly converges to $\ell$ in $[0,T]$), then the sequence of auxiliary functions $\{\vk\}_{k\in\enne}$ converges to $v$ in the following ways:
	\begin{equation}\label{allconv}
		\begin{aligned}
			&\bullet\vk\to v\mbox{ uniformly in }[0,T]\times[0,+\infty);\\
			&\bullet\vk\to v\mbox{ in }H^1((0,T)\times(0,+\infty));\\
			&\bullet\vk\to v\mbox{ in }C^0([0,T];H^1(0,+\infty))\mbox{ and in } C^1([0,T];L^2(0,+\infty));\\
			&\bullet\vkx(\cdot,0)\to v_x(\cdot,0)\mbox{ and }\sqrt{1-\elldk(\cdot)^2}\vkx(\cdot,\ellk(\cdot))\to \sqrt{1-\elld(\cdot)^2}v_x(\cdot,\ell(\cdot))\mbox{ in }L^2(0,T).
		\end{aligned}
	\end{equation}
\begin{rmk}\label{auxiliar}
	We recall that by the formula $u^k(t,x)=e^{-\nuk t/2}\vk(t,x)$ we deduce that all the convergences in \eqref{allconv} still remains true replacing $v^k$ and $v$ by the real solutions of the coupled problem $u^k$ and $u$ respectively.
\end{rmk}

	\section{The continuous dependence result}\label{sec3}
	The goal of this Section is proving that under assumptions \eqref{conv} there exists a small time $\overline T>0$ such that $\elldk\to\elld$ in $L^1(0,\overline T)$. In this case, by what we proved in Section \ref{sec2}, we will deduce as a byproduct that all the convergences in \eqref{allconv} hold true in $[0,\overline T]$. This will lead us to the main Theorem~of the paper, namely Theorem~\ref{finalthm}, in which we extend the result to arbitrary large time.
	
	To this aim, as in \cite{DMLazNar16} and \cite{RivNar}, we introduce the functions $\lambdak$ and $\lambda$ as the inverse of $\varphik$ and $\varphi$, respectively. By \eqref{repformula}, \eqref{Gamma} and by using the classical formula for the derivative of inverse functions we deduce that for $T<\frac{\ellz}{2}$ we can write:
	\begin{equation}\label{lambda}
		\lambdak(y)=\frac 12 \int_{-\ellzk}^{y}\left(1+\max\left\{\Theta_{\vk,\lambdak}^k(s),1\right\}\right)\d s,\quad\mbox{for every }y\in[-\ellzk,\varphik(T)],
	\end{equation}
	where for a.e. $y\in [-\ellzk,\varphik(T)]$ we considered the function:
	\begin{equation}\label{Lambdak}
		\Theta_{\vk,\lambdak}^k(y)=\frac{\left[\vzdk(-y)-\vuk(-y)-\frac{(\nuk)^2}{4}\int_{0}^{\lambdak(y)}\vk(\tau,\tau{-}y)\d\tau\right]^2}{2e^{\nuk\lambdak(y)}\kappak(\lambdak(y){-}y)}.
	\end{equation}
	Obviously the same formulas without apexes $k$ hold true also for $\lambda$.\\ 
	Furthermore let us define the set (see Figure~\ref{FigQs}): 
	\begin{equation*}
		Q^k:=\left\{(t,x)\in \erre^2\mid t\in[0,T]\mbox{ and } x\in[t-(\varphi(T)\wedge\varphik(T)),t+(\ellz\wedge\ellzk)] \right\},
	\end{equation*}
	and let us introduce the distance:
	\begin{equation}\label{distance}
		\d\left((v^k,\lambdak),(v,\lambda)\right):=\max\Big\{\Vert v^k-v\Vert_{L^2(Q^k)}, \max\limits_{y\in D^k_\varphi(0,T)}|\lambdak(y)-\lambda(y)|\Big\}.
	\end{equation}
	\begin{rmk}
		This distance is the analogue in our context of the one used in \cite{RivNar} to show that a certain operator (the right-hand side of representation formulas for $\vk$ and $\lambdak$, see \eqref{duhamv} and \eqref{lambda}) is a contraction in a suitable space. This will help us to reach our goal.
	\end{rmk}
	First of all let us prove that $D^k_\varphi(0,T)=[-(\ellz\wedge\ellzk),\varphi(T)\wedge\varphi^k(T)]$ is definitively nondegenerate.
	\begin{lemma}\label{nonempty} 
		Assume \eqref{conv} and let $T$ be as in Proposition~\ref{bdd}. Then there exists $K\in\enne$ such that for every $k\geq K$ the set $D^k_\varphi(0,T)$ is a nondegenerate closed interval.		
	\end{lemma}
\begin{proof}
	We argue by contradiction. Let us assume that there exists a subsequence (not relabelled) such that $D^k_\varphi(0,T)$ is empty or it is a singleton for every $k\in\enne$. Since $\ellzk\to\ellz$ and since $\varphi(T)>-\ell_0$ we can exclude the case $\varphi(T)\le-\ellzk<\varphik(T)$ for every $k$. This means that for every $k\in\enne$ we have $-\ellzk<\varphik(T)\leq-\ellz$ .\par
	\textit{CLAIM.} We claim that in this case $\lim\limits_{k\to +\infty}\max\limits_{y\in [-\ellzk,\varphik(T)]}|\lambdak(y)|=0$.\par
	If the claim is true we conclude; indeed by definition $\lambdak(\varphik(T))=T$ and hence we get a contradiction.\\
	To prove the claim we fix $y\in [-\ellzk,\varphik(T)]$ and we estimate:
	\begin{align*}
		\lambdak(y)&\leq\frac 12 \int_{-\ellzk}^{\varphik(T)}\left(1+\max\left\{\Theta_{\vk,\lambdak}^k(s),1\right\}\right)\d s \leq \int_{-\ellzk}^{\varphik(T)}\left(1+\frac 12 \Theta_{\vk,\lambdak}^k(s)\right)\d s\\
		&=\varphik(T)+\ellzk+\frac 14 \int_{-\ellzk}^{\varphik(T)}\frac{\left[\vzdk(-s)-\vuk(-s)-\frac{(\nuk)^2}{4}\int_{0}^{\lambdak(s)}\vk(\tau,\tau{-}s)\d\tau\right]^2}{e^{\nuk\lambdak(s)}\kappak(\lambdak(s){-}s)}\d s.
	\end{align*}
	Since $-\ellzk<\varphik(T)\leq -\ellz$, by \eqref{convlnu} we deduce that $\varphik(T)+\ellzk\to 0$ as $k\to+\infty$. Then we estimate the integral in the last line exploiting Proposition~\ref{bdd} and hypothesis \eqref{convkappa}:
	\begin{align*}
		&\quad\,\int_{-\ellzk}^{\varphik(T)}\frac{\left[\vzdk(-s)-\vuk(-s)-\frac{(\nuk)^2}{4}\int_{0}^{\lambdak(s)}\vk(\tau,\tau{-}s)\d\tau\right]^2}{e^{\nuk\lambdak(s)}\kappak(\lambdak(s){-}s)}\d s\\
		&\leq C \int_{-\ellzk}^{\varphik(T)}\left(\vzdk(-s)^2+\vuk(-s)^2+N^4M^2T^2\right)\d s=C \int_{-\varphik(T)}^{\ellzk}\left(\vzdk(s)^2+\vuk(s)^2+1\right)\d s.
	\end{align*}
	By hypothesis \eqref{convdata} and since $\varphik(T)+\ellzk\to 0$ we conclude.
\end{proof}
To make next Proposition clearer let us introduce for $y<0$ the functions $j^k(y):=|\dot{v}^k_0(-y)|+|v_1^k(-y)|+\chi_{[0,2\ellz]}(-y)$ and notice that by \eqref{convdata} the sequence $\{j^k\}_{k\in\enne}$ is equibounded in $L^2(-\infty,0)$. Here $\chi_{[0,2\ellz]}$ stands for the characteristic function of $[0,2\ell_0]$; the choice of such an interval is simply related to the fact that definitively $0<\ellzk<2\ellz$, since $\ellzk\to\ellz$ as $k\to +\infty$, and thus $D^k_\varphi(0,T)\subseteq[-2\ell_0,0]$ if the time $T$ is small enough. Moreover, to simplify the expression of $\Theta_{\vk,\lambdak}^k$ in \eqref{Lambdak}, we also define the functions $\displaystyle\rho^k(y):=\dot v_0^k({-}y)-v_1^k({-}y)-\frac{(\nuk)^2}{4}\int_{0}^{\lambdak(y)}\!\!\!\!\!v^k(\tau,\tau{-}y)\d\tau$ and using Proposition~\ref{bdd} we observe that 
\begin{equation}\label{inequality}
	|\rho^k(y)|\le C j^k(y),\quad\text{for a.e. }y\in D^k_\varphi(0,T),
\end{equation}
if the time $T$ is sufficiently small. In the same way we define the functions $j$ and $\rho$. Finally we introduce the nonnegative quantity:
\begin{equation}\label{etak}
	\eta^k:=\Vert j\Vert^2_{L^2(D^k_\varphi(0,T))}+\Vert j^k\Vert_{L^2(D^k_\varphi(0,T))}+\Vert j\Vert_{L^2(D^k_\varphi(0,T))}.
\end{equation}

\begin{prop}\label{estimatelprop}
	Assume \eqref{conv}, let $T$ be as in Proposition~\ref{bdd} and let $K$ be given by Lemma \ref{nonempty}. Then there exists a constant $C_1\geq 0$ independent of $k$ and an infinitesimal sequence $\{\eps^k\}_{k\in\enne}$ such that for every $k\geq K$ the following estimate  holds true:
	\begin{equation}\label{estimatel}
		\max\limits_{y\in D^k_\varphi(0,T)}|\lambdak(y)-\lambda(y)|\leq \eps^k+C_1\eta^k \d\left((v^k,\lambdak),(v,\lambda)\right).
	\end{equation}
\end{prop}
	\begin{proof}
		We assume $\ellz<\ellzk$, being the other cases even easier, and we estimate by means of \eqref{lambda} and \eqref{Lambdak}:
		\begin{equation}\label{important}
			\begin{aligned}
			&\max\limits_{y\in D^k_\varphi(0,T)}|\lambdak(y)-\lambda(y)|\\
			&\leq\int_{-\ellzk}^{-\ellz}\lambdadk(s)\d s+\frac 14\int_{D^k_\varphi(0,T)}\left|\frac{\rho^k(s)^2}{e^{\nuk\lambdak(s)}\kappak(\lambdak(s){-}s)}-\frac{\rho(s)^2}{e^{\nu\lambda(s)}\kappa(\lambda(s){-}s)}\right|\d s.
			\end{aligned}
		\end{equation}		
		The first term goes to zero as $k\to+\infty$ reasoning as in the proof of Lemma~\ref{nonempty}. For the second one, denoted by $I^k$, we estimate by using triangular inequality and exploiting assumption \eqref{convkappa} to get uniform bounds on $\kappak$:
		\begingroup
		\allowdisplaybreaks
		\begin{align*}
			I^k&\leq C\!\!\int_{D^k_\varphi(0,T)}\!\!\!\!\!\!\!\!\!\!\!\!\!\!\!\!e^{\nu\lambda(s)}\kappa(\lambda(s){-}s)\left|\rho^k(s)^2{-}\rho(s)^2\right|\!\!\d s+C\!\!\int_{D^k_\varphi(0,T)}\!\!\!\!\!\!\!\!\!\!\!\!\!\!\!\!\rho(s)^2\left|e^{\nuk\lambdak(s)}\kappak(\lambdak(s){-}s){-}e^{\nu\lambda(s)}\kappa(\lambda(s){-}s)\right|\!\!\d s\\
			&\leq C\int_{D^k_\varphi(0,T)}\left|\rho^k(s)^2-\rho(s)^2\right|\d s+C\int_{D^k_\varphi(0,T)}\rho(s)^2\left|e^{(\nuk-\nu)\lambdak(s)}-1\right|\d s\\
			&\quad+C\max\limits_{y\in [0,\ellz+T]}|\kappak(y)-\kappa(y)|\int_{D^k_\varphi(0,T)}\rho(s)^2\d s+C\max\limits_{y\in D^k_\varphi(0,T)}|\lambdak(y)-\lambda(y)|\int_{D^k_\varphi(0,T)}\rho(s)^2\d s .
		\end{align*}
		\endgroup
		By dominated convergence and by \eqref{convlnu} and \eqref{convkappa} the second and the third term go to zero as $k\to+\infty$, while for the first term we estimate by using the explicit expressions of $\rho^k$ and $\rho$ and recalling \eqref{inequality}:
		\begingroup
		\allowdisplaybreaks
		\begin{align*}
			&\quad\,\int_{D^k_\varphi(0,T)}\left|\rho^k(s)^2-\rho(s)^2\right|\d s\\
			&\leq \int_{D^k_\varphi(0,T)}\!\!\!\!\!\!\!\!|\vzdk({-}s)-\vzd({-}s)|\left(|\rho^k(s)|+|\rho(s)|\right)\d s+\int_{D^k_\varphi(0,T)}\!\!\!\!\!\!\!\!|\vuk({-}s)-\vu({-}s)|\left(|\rho^k(s)|+|\rho(s)|\right)\d s\\
			&\quad+\frac 14 \int_{D^k_\varphi(0,T)}(|\rho^k(s)|+|\rho(s)|)\left|(\nuk)^2\int_{0}^{\lambdak(s)}\vk(\tau,\tau{-}s)\d\tau-\nu^2\int_{0}^{\lambda(s)}v(\tau,\tau{-}s)\d\tau\right|\d s\\
			&\leq C \left(\Vert\vzdk-\vzd\Vert_{L^2(0,+\infty)}+\Vert\vuk-\vu\Vert_{L^2(0,+\infty)}\right)\left(\Vert j^k\Vert_{L^2(-\infty,0)}+\Vert j\Vert_{L^2(-\infty,0)}\right)\\
			&\quad+C\int_{D^k_\varphi(0,T)}(|j^k(s)|+|j(s)|)\left|(\nuk)^2\int_{0}^{\lambdak(s)}\vk(\tau,\tau{-}s)\d\tau-\nu^2\int_{0}^{\lambda(s)}v(\tau,\tau{-}s)\d\tau\right|\d s.
		\end{align*}
		\endgroup
		To deal with the last integral we first notice that for every $s\in D^k_\varphi(0,T)$ we have:
		\begingroup
		\allowdisplaybreaks
		\begin{align*}
			&\quad\,\left|(\nuk)^2\int_{0}^{\lambdak(s)}\vk(\tau,\tau{-}s)\d\tau-\nu^2\int_{0}^{\lambda(s)}v(\tau,\tau{-}s)\d\tau\right|\\
			&\leq|(\nuk)^2-\nu^2|\left|\int_{0}^{\lambdak(s)}\!\!\!\!v^k(\tau,\tau{-}s)\d\tau\right|+\nu^2\left|\int_{0}^{\lambdak(s)}\!\!\!\!(v^k-v)(\tau,\tau{-}s)\d\tau\right|+\nu^2\left|\int_{\lambda(s)}^{\lambdak(s)}\!\!\!\!v(\tau,\tau{-}s)\d\tau\right|\\
			&\leq C\left(|(\nuk)^2-\nu^2|+\left|\int_{0}^{T}(v^k-v)(\tau,\tau{-}s)\d\tau\right|+\max\limits_{y\in D^k_\varphi(0,T)}|\lambdak(y)-\lambda(y)|\right),
		\end{align*}
		\endgroup
		and so we deduce:
		\begin{align*}
			&\quad\max\limits_{y\in D^k_\varphi(0,T)}|\lambdak(y)-\lambda(y)|\leq \eps^k+I^k\\
			&\leq \eps^k+C\left(\Vert j\Vert^2_{L^2(D^k_\varphi(0,T))}+\Vert j^k\Vert_{L^2(D^k_\varphi(0,T))}+\Vert j\Vert_{L^2(D^k_\varphi(0,T))}\right) \d\left((v^k,\lambdak),(v,\lambda)\right)\\
			&= \eps^k+C\eta^k \d\left((v^k,\lambdak),(v,\lambda)\right),
		\end{align*}
		and we conclude.
	\end{proof}
	\begin{prop}\label{estimatevprop}
		Assume \eqref{conv}, let $T$ be as in Proposition~\ref{bdd} and let $K$ be given by Lemma \ref{nonempty}. Then there exists a constant $C_2\geq 0$ independent of $k$ and an infinitesimal sequence $\{\eps^k\}_{k\in\enne}$ such that for every $k\geq K$ the following estimate  holds true: 
		\begin{equation}\label{estimatev}
			\Vert v^k-v\Vert_{L^2(Q^k)}\leq \eps^k+C_2\sqrt{|D^k_\varphi(0,T)|}\d\left((v^k,\lambdak),(v,\lambda)\right).
		\end{equation}
	\end{prop}
	\begin{proof}
		\begin{figure}
			\centering
			\includegraphics[scale=.8]{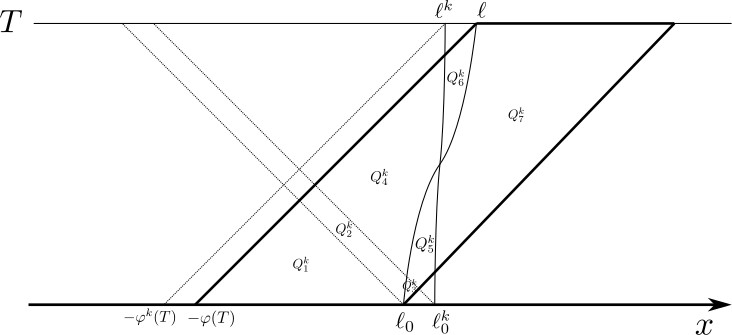}
			\caption{The partition of the set $Q^k$ via the sets $Q^k_i$, for $i=1,\dots,7$, in the case $\ellz<\ellzk$ and $\varphi(T)<\varphik(T)$.}\label{FigQs}
		\end{figure}
		We use again formula \eqref{duhamv} and we estimate:
		\begin{equation}\label{firstest}
			\begin{aligned}
			\Vert v^k-v\Vert_{L^2(Q^k)}&\leq\Vert \Ak-A\Vert_{L^2(Q^k)}+\frac{(\nuk)^2}{8}\Vert \Hk-H\Vert_{L^2(Q^k)}+\frac{|(\nuk)^2-\nu^2|}{8}\Vert H\Vert_{L^2(Q^k)}\\
			&\leq \eps^k+\Vert \Ak-A\Vert_{L^2(Q^k)}+\frac{N^2}{8}\Vert \Hk-H\Vert_{L^2(Q^k)}.
			\end{aligned}
		\end{equation}		
		Then we split $Q^k$ into seven parts, denoted by $Q^k_i$ for $i=1,\dots,7$, as in Figure \ref{FigQs}, so that: 
		\begin{equation}\label{splitA}
			\Vert \Ak-A\Vert^2_{L^2(Q^k)}\!=\!\!\iint_{Q_1^k\cup Q_2^k\cup Q^k_4}\!\!\!\!\!\!\!\!\!\!\!\!\!\!\!\!\!\!\!\!|\Ak(t,x)-A(t,x)|^2\!\d x\d t+\iint_{Q_3^k\cup Q_5^k}\!\!\!\!\!\!\!\!\!\!\!\!\Ak(t,x)^2\!\d x\d t+\iint_{Q_6^k}\!\!\!\!A(t,x)^2\!\d x\d t,
		\end{equation}
		and we estimate all the terms.\par 
		The integrals over $Q^k_1$, $Q^k_2$, $Q^k_3$ go easily to zero as $k\to+\infty$: indeed in $Q^k_1$ we use \eqref{convdata}, while for the integrals over $Q_2^k$ and $Q_3^k$ we exploit the equiboundedness of the sequence $\{\Ak\}_{k\in\enne}$ in $C^0([0,T]\times[0,+\infty)$ (see Proposition~\ref{bdd}) and the fact that $Q_2^k\cup Q^k_3$ converges in measure to the empty set. To estimate the remaining terms we reason as in \cite{RivNar}, Proposition~4.5. In that work the validity of the following estimates is proved:
		\begin{subequations}
			\begin{equation}\label{supa}
			|\omega^k(x{+}t)-\omega(x{+}t)|\le 2\max\limits_{y\in D^k_\varphi(0,T)}|\lambdak(y)-\lambda(y)|,\quad\text{if }(t,x)\in Q^k_4,
			\end{equation}
			\begin{equation}\label{supb}
			|(t{-}x)-\omega^k(x{+}t)|\le 2\max\limits_{y\in D^k_\varphi(0,T)}|\lambdak(y)-\lambda(y)|,\quad\text{if }(t,x)\in Q^k_5,
			\end{equation}
			\begin{equation}\label{supc}
			|(t{-}x)-\omega(x{+}t)|\le 2\max\limits_{y\in D^k_\varphi(0,T)}|\lambdak(y)-\lambda(y)|,\quad\text{if }(t,x)\in Q^k_6.
			\end{equation}
		\end{subequations}
	Moreover they also show that:
	\begin{equation}\label{cube}
		|Q^k_5\cup Q^k_6|\leq |D^k_\varphi(0,T)|\max\limits_{y\in D^k_\varphi(0,T)}|\lambdak(y)-\lambda(y)|.
	\end{equation}
	Exploiting \eqref{supb}, \eqref{supc}, \eqref{cube} and reasoning as in the proof of Proposition~4.5 in \cite{RivNar} one can deduce that: 
	\begin{equation*}
		\iint_{Q_5^k}\Ak(t,x)^2\d x\d t+\iint_{Q_6^k}A(t,x)^2\d x\d t\leq C|D^k_\varphi(0,T)|\max\limits_{y\in D^k_\varphi(0,T)}|\lambdak(y)-\lambda(y)|^2.
	\end{equation*}
	To estimate the integral over $Q^k_4$ we first of all notice that for $(t,x)\in Q^k_4$ we have:
	\begingroup
	\allowdisplaybreaks
	\begin{align*}
		|\Ak(t,x)-A(t,x)|^2&=\frac 14\left|\int_{x{-}t}^{-\omega^k(x{+}t)}(\vuk(s)-\vzdk(s))\d s-\int_{x{-}t}^{-\omega(x{+}t)}(\vu(s)-\vzd(s))\d s\right|^2\\
		&\leq C\left(\Vert\vuk-\vu\Vert^2_{L^2(0,+\infty)}+\Vert\vzdk-\vzd\Vert^2_{L^2(0,+\infty)}\right) +\frac 12 \left|\int_{-\omega(x{+}t)}^{-\omega^k(x{+}t)}\!\!\!\!\!\!\!\!\!\!\!\!\!\!\!\!(\vu(s)-\vzd(s))\d s\right|^2\\
		&\leq \eps^k+\frac 12 |\omega^k(x{+}t)-\omega(x{+}t)|\left|\int_{-\omega(x{+}t)}^{-\omega^k(x{+}t)}(\vu(s)-\vzd(s))^2\d s\right|.
	\end{align*}
	\endgroup
	Using \eqref{supa} we then deduce that for $(t,x)\in Q^k_4$ the following estimate holds true:
	\begin{equation*}
		|\Ak(t,x)-A(t,x)|^2\leq \eps^k+\max\limits_{y\in D^k_\varphi(0,T)}|\lambdak(y)-\lambda(y)|\left|\int_{-\omega(x{+}t)}^{-\omega^k(x{+}t)}(\vu(s)-\vzd(s))^2\d s\right|.
	\end{equation*}
	From this inequality, reasoning as in the proof of Proposition~4.5 in \cite{RivNar}, we conclude that: 
	\begin{equation*}
		\iint_{Q^k_4}|\Ak(t,x)-A(t,x)|^2\d x\d t\leq \eps^k+C|D^k_\varphi(0,T)|\max\limits_{y\in D^k_\varphi(0,T)}|\lambdak(y)-\lambda(y)|^2.
	\end{equation*}
	Putting all the previous estimates together we deduce:
	\begin{equation}\label{secondest}
	\begin{aligned}
	\Vert \Ak-A\Vert^2_{L^2(Q^k)}&\leq \eps^k+C|D^k_\varphi(0,T)|\max\limits_{y\in D^k_\varphi(0,T)}|\lambdak(y)-\lambda(y)|^2\\
	&\leq \eps^k+C|D^k_\varphi(0,T)|\d\left((v^k,\lambdak),(v,\lambda)\right)^2.
	\end{aligned}
	\end{equation}
	Now we estimate $\Vert \Hk-H\Vert_{L^2(Q^k)}$. As in \eqref{splitA} we split its square into six integrals  and we estimate all of them. With the same argument used before we deduce the integral over $Q^k_2\cup Q^k_3$ goes to zero as $k\to+\infty$, while the integral over $Q^k_1$ is trivially bounded by $C|D^k_\varphi(0,T)|\Vert \vk-v\Vert^2_{L^2(Q^k)}$. More work is needed to treat the other three integrals. Exploiting Proposition~\ref{bdd} we estimate the integrals over $Q^k_5$ and $Q^k_6$ together:
	\begingroup
	\allowdisplaybreaks
	\begin{align*}
		&\quad\,\iint_{Q_5^k}\Hk(t,x)^2\d x\d t+\iint_{Q_6^k}H(t,x)^2\d x\d t\\
		&\leq C\left(\iint_{Q_5^k}|R^k(t,x)|^2\d x\d t+\iint_{Q_6^k}|R(t,x)|^2\d x\d t\right)\\
		&\leq C\left(\iint_{Q_5^k}|(t{-}x)-\omega^k(x{+}t)|^2\d x\d t+\iint_{Q_6^k}|(t{-}x)-\omega(x{+}t)|^2\d x\d t\right).
	\end{align*} 
	\endgroup
	So, using \eqref{supb} and \eqref{supc} we deduce:
	\begin{equation*}
	\iint_{Q_5^k}\Hk(t,x)^2\d x\d t+\iint_{Q_6^k}H(t,x)^2\d x\d t\leq C|D^k_\varphi(0,T)|\max\limits_{y\in D^k_\varphi(0,T)}|\lambdak(y)-\lambda(y)|^2.
	\end{equation*} 
	For the integral over $Q^k_4$ we use \eqref{supa} and we reason as follows:
	\begingroup
	\allowdisplaybreaks
	\begin{align*}
		&\quad\,\iint_{Q_4^k}|\Hk(t,x)-H(t,x)|^2\d x\d t\\
		&\leq \iint_{Q_4^k}\left(\iint_{R^k(t,x)}|\vk(\tau,\sigma)-v(\tau,\sigma)|\d\sigma\d\tau+\iint_{R^k(t,x)\Delta R(t,x)}\!\!\!\!\!\!\!\!\!\!\!\!|v(\tau,\sigma)|\d\sigma\d\tau\right)^2\d x\d t\\
		&\leq C\iint_{Q_4^k}\left(|R^k(t,x)|\Vert\vk-v\Vert^2_{L^2(Q^k)}+|R^k(t,x)\Delta R(t,x)|^2\right)\d x\d t\\
		&\leq C\left(|D^k_\varphi(0,T)|\Vert\vk-v\Vert^2_{L^2(Q^k)}+\iint_{Q_4^k}|\omega^k(x{+}t)-\omega(x{+}t)|^2\d x\d t\right)\\
		&\leq C|D^k_\varphi(0,T)|\left(\Vert\vk-v\Vert^2_{L^2(Q^k)}+\max\limits_{y\in D^k_\varphi(0,T)}|\lambdak(y)-\lambda(y)|^2\right).
	\end{align*}
	\endgroup
	Putting together the previous estimates we conclude that: 
	\begin{equation}\label{thirdest}
	\begin{aligned}
	\Vert \Hk-H\Vert^2_{L^2(Q^k)}&\leq \eps^k+C|D^k_\varphi(0,T)|\left(\Vert\vk-v\Vert^2_{L^2(Q^k)}+\max\limits_{y\in D^k_\varphi(0,T)}|\lambdak(y)-\lambda(y)|^2\right)\\
	&\leq \eps^k+C|D^k_\varphi(0,T)|\d\left((v^k,\lambdak),(v,\lambda)\right)^2,
	\end{aligned}
	\end{equation}
	and so by \eqref{firstest}, \eqref{secondest} and \eqref{thirdest} the Proposition~is proved.
	\end{proof}\noindent
 	Putting together \eqref{estimatel} and \eqref{estimatev} we deduce that there exists a constant $\overline C\geq 0$ independent of $k$ such that for every $k$ large enough it holds:
 	\begin{equation}\label{estimated}
 	\d\left((v^k,\lambdak),(v,\lambda)\right)\leq \eps^k+\overline C\max\left\{\eta^k, |D^k_\varphi(0,T)| \right\}\d\left((v^k,\lambdak),(v,\lambda)\right).
 	\end{equation}
 	By \eqref{estimated} we are able to improve Lemma~\ref{nonempty}:
 	\begin{lemma}\label{newlemma}
 			Assume \eqref{conv} and let $T$ be as in Proposition~\ref{bdd}. Then there exist $\delta>0$ and $\widetilde{K}\in\enne$ such that for every $k\geq \widetilde{K}$ the nondegenerate closed interval $J^k_\delta=[-(\ellzk\wedge\ellz),-\ellz+\delta]$ is contained in $D^k_\varphi(0,T)$.
 \end{lemma}
\begin{proof}
	Assume by contradiction that there exists a subsequence (not relabelled) such that $D^k_\varphi(0,T)$ goes to the empty set when $k\to +\infty$, namely $\lim\limits_{k\to +\infty}\varphik(T)=-\ellz$. By \eqref{etak} we in particular deduce that $\eta^k\to 0$ as $k\to +\infty$.\par
	By \eqref{estimated} we thus infer $\lim\limits_{k\to +\infty}\d\left((v^k,\lambdak),(v,\lambda)\right)=0$, which in particular implies:
	\begin{equation*}
		\lim\limits_{k\to +\infty}\max\limits_{y\in [-(\ellzk\wedge\ellz),\varphik(T)]}|\lambdak(y)-\lambda(y)|=0.
	\end{equation*}
	This is absurd, indeed:
	\begin{equation*}
		\lim\limits_{k\to +\infty}|\lambdak(\varphik(T))-\lambda(\varphik(T))|=\lim\limits_{k\to +\infty}|T-\lambda(\varphik(T))|=|T-\lambda(-\ellz)|=T>0,
	\end{equation*}
	and we conclude.
\end{proof}
From this Lemma, repeating the proofs of Propositions~\ref{estimatelprop} and \ref{estimatevprop} we deduce that \eqref{estimated} still holds true replacing $D^k_\varphi(0,T) $ by $J_\delta^k$, replacing $\eta^k$ by $\eta^k_\delta:=\Vert j\Vert^2_{L^2(J^k_\delta)}+\Vert j^k\Vert_{L^2(J^k_\delta)}+\Vert j\Vert_{L^2(J^k_\delta)}$ and replacing $Q^k$ by $Q^k_\delta:=\left\{(t,x)\in \erre^2\mid t\in[0,T]\mbox{ and } x\in[t+\ellz-\delta,t+(\ellz\wedge\ellzk)] \right\}$. This means that, choosing $\delta$ small enough, for every $k$ large enough we have:
 	\begin{equation}\label{distdelta}
 	\d_\delta\left((v^k,\lambdak),(v,\lambda)\right)\leq \eps^k+\frac 12\d_\delta\left((v^k,\lambdak),(v,\lambda)\right),
 	\end{equation}
 	where the new distance $\d_\delta$ is simply as in \eqref{distance} replacing $D^k_\varphi(0,T) $ by $J_\delta^k$ and $Q^k$ by $Q^k_\delta$.
 	By \eqref{distdelta} we finally deduce that:
 	\begin{equation}\label{limddelta}
 		\lim\limits_{k\to +\infty}\d_\delta\left((v^k,\lambdak),(v,\lambda)\right)=0.
 	\end{equation}
 	Furthermore by \eqref{limddelta} we get:
 	\begin{equation}\label{w11est}
 		\lim\limits_{k\to +\infty}\int_{-(\ellzk\wedge\ellz)}^{-\ellz+\delta} |\lambdadk(y)-\lambdad(y)|\d y=0.
 	\end{equation}
 	To justify the validity of \eqref{w11est} we reason as follows: in the estimate \eqref{important} at the beginning of the proof of Proposition~\ref{estimatelprop} we can replace $\max\limits_{y\in J^k_\delta}|\lambdak(y)-\lambda(y)|$ by $\displaystyle \int_{-(\ellzk\wedge\ellz)}^{-\ellz+\delta} |\lambdadk(y)-\lambdad(y)|\d y$, obtaining that: 
 	\begin{equation*}
 		\int_{-(\ellzk\wedge\ellz)}^{-\ellz+\delta} |\lambdadk(y)-\lambdad(y)|\d y\leq \eps^k+C_1\eta^k_\delta\d_\delta\left((v^k,\lambdak),(v,\lambda)\right),
 	\end{equation*}
 	and so by \eqref{limddelta} we conclude the argument. This leads to the following Corollary:
 	\begin{cor}\label{time}
 		Assume \eqref{conv}. Then there exists a small time $\overline T>0$ such that $\elldk\to\elld$ in $L^1(0,\overline T)$.
 	\end{cor}
 	\begin{proof}
 		Let us take any $\overline T\in\big(0,\lambda(-\ellz+\delta)\big)$, where $\delta$ is given by Lemma~\ref{newlemma} and such that \eqref{w11est} holds true, and for the sake of clarity let us consider the value $m^k:=\lambdak\big({-}(\ellzk\wedge \ellz)\big)\vee\lambda\big({-}(\ellzk\wedge \ellz)\big)$. Then we have:
 		\begin{align*}
 			\Vert\elldk-\elld\Vert_{L^1(0,\overline T)}&=\int_{0}^{m^k}|\elldk(s)-\elld(s)|\d s+\int_{m^k}^{\overline T}|\elldk(s)-\elld(s)|\d s\\
 			&\leq 2m^k+\int_{m^k}^{\overline T}\left|\frac{1}{\lambdadk(\lambdak^{-1}(s))}-\frac{1}{\lambdad(\lambda^{-1}(s))}\right|\d s.
 		\end{align*}
 		By uniform convergence of $\lambdak$ to $\lambda$ and by \eqref{convlnu} the first term goes to zero as $k\to+\infty$, while for the second one, denoted by $I^k$, we estimate:
 		\begin{align*}
 			I^k&\leq \int_{m^k}^{\overline T}\left|\frac{\lambdadk(\lambdak^{-1}(s))-\lambdad(\lambdak^{-1}(s))}{\lambdadk(\lambdak^{-1}(s))}\right|\d s+\int_{m^k}^{\overline T}\left|\frac{\lambdad(\lambdak^{-1}(s))-\lambdad(\lambda^{-1}(s))}{\lambdadk(\lambdak^{-1}(s))\lambdad(\lambda^{-1}(s))}\right|\d s\\
 			&\leq \int_{-(\ellzk\wedge\ellz)}^{-\ellz+\delta}\left|\lambdadk(y)-\lambdad(y)\right|\d y+\int_{m^k}^{\overline T}\left|\frac{\lambdad(\lambdak^{-1}(s))-\lambdad(\lambda^{-1}(s))}{\lambdadk(\lambdak^{-1}(s))\lambdad(\lambda^{-1}(s))}\right|\d s.
 		\end{align*}
 		By \eqref{w11est} the first term goes to zero as $k\to +\infty$; for the second one, denoted by $II^k$, we reason as follows: we fix $\eps>0$ and we take $f_\eps\in C^0([-\ell_0,-\ellz+\delta])$ such that $\Vert\lambdad-f_\eps\Vert_{L^1(-\ellz,-\ellz+\delta)}\leq \eps$. Then we estimate:
 		\begingroup
 		\allowdisplaybreaks
 		\begin{align*}
 			II^k&\leq \int_{m^k}^{\overline T}\left|\frac{\lambdad(\lambdak^{-1}(s))-f_\eps(\lambdak^{-1}(s))}{\lambdadk(\lambdak^{-1}(s))}\right|\d s+\int_{m^k}^{\overline T}|f_\eps(\lambdak^{-1}(s))-f_\eps(\lambda^{-1}(s))|\d s\\
 			&\quad+\int_{m^k}^{\overline T}\left|\frac{f_\eps(\lambda^{-1}(s))-\lambdad(\lambda^{-1}(s))}{\lambdad(\lambda^{-1}(s))}\right|\d s\\
 			&\leq2\int_{-\ellz}^{-\ellz+\delta}\left|\lambdad(y)-f_\eps(y)\right|\d y+\int_{m^k}^{\overline T}|f_\eps(\lambdak^{-1}(s))-f_\eps(\lambda^{-1}(s))|\d s\\\
 			&\leq 2\eps+\int_{m^k}^{\overline T}|f_\eps(\lambdak^{-1}(s))-f_\eps(\lambda^{-1}(s))|\d s.
 		\end{align*}
 		\endgroup
 		By Lemma~\ref{unifinv} and dominated convergence this last integral vanishes as $k\to+\infty$, hence by the arbitrariness of $\eps$ we conclude.
 	\end{proof}
	We are now in a position to state and prove the main result of the paper:
	\begin{thm}\label{finalthm}
		Assume \eqref{conv}. Then the sequence of pairs $\{(u^k,\ell^k)\}_{k\in\enne}$ converges to the solution of the limit problem $(u,\ell)$ in the following sense: for every $T>0$
		\begin{equation}\label{thesis}
		\begin{aligned}
		&\bullet\elldk\to\elld \mbox{ in } L^1(0,T), \mbox{ and thus }\ellk\to\ell\mbox{ uniformly in }[0,T];\\
		&\bullet u^k\to u\mbox{ uniformly in }[0,T]\times[0,+\infty);\\
		&\bullet u^k\to u\mbox{ in }H^1((0,T)\times(0,+\infty));\\
		&\bullet u^k\to u\mbox{ in }C^0([0,T];H^1(0,+\infty))\mbox{ and in } C^1([0,T];L^2(0,+\infty));\\
		&\bullet u^k_x(\cdot,0)\to u_x(\cdot,0)\mbox{ and }\sqrt{1-\elldk(\cdot)^2}u^k_x(\cdot,\ellk(\cdot))\to \sqrt{1-\elld(\cdot)^2}u_x(\cdot,\ell(\cdot))\mbox{ in }L^2(0,T).
		\end{aligned}
		\end{equation}		
	\end{thm}
	\begin{proof}
		As already remarked previously it is enough to prove that \eqref{thesis} holds true for the sequence of auxiliary functions $\vk(t,x)=e^{\nuk t/2}u^k(t,x)$. By Corollary \ref{time} and by the results presented in Section \ref{sec2} we know there exists a small time $\overline T>0$ such that all the convergences in \eqref{thesis} hold true in $[0,\overline T]$ for the sequence of pairs $\{(\vk,\ellk)\}_{k\in\enne}$. So we can consider:
		\begin{equation*}
			T^*:=\sup\{\overline T>0\mid (\vk,\ellk)\to(v,\ell)\mbox{ in the sense of }\eqref{thesis}\mbox{ in }[0,\overline T]\}.
		\end{equation*}
		If $T^*=+\infty$ we conclude. So let us argue by contradiction assuming that $T^*$ is finite. This means there exists an increasing sequence of times $\{T^j\}_{j\in\enne}$ converging to $T^*$ and for which $(\vk,\ellk)\to(v,\ell)$ in the sense of \eqref{thesis} in $[0,T^j]$ for every $j\in\enne$. Since $\elldk\to\elld$ in $L^1(0,T^j)$ for every $j\in\enne$ and $\elldk(t)<1$ and $\elld(t)<1$ for a.e. $t>0$ it follows that $\elldk\to\elld$ in $L^1(0,T^*)$ and hence $\ellk$ uniformly converges to $\ell$ in $[0,T^*]$ by \eqref{convlnu}. Moreover, reasoning as in Section \ref{sec2} we also get that $\vk\to v$ in the sense of \eqref{thesis} in the whole time interval $[0,T^*]$, and hence $T^*$ is a maximum. Now we can repeat the proofs of Propositions~\ref{estimatelprop} and \ref{estimatevprop} starting from time $T^*$(notice that by \eqref{thesis} the convergence hypothesis \eqref{convdata} is fulfilled by $u^k(T^*,\cdot)$ and $u^k_t(T^*,\cdot)$, while \eqref{convlnu} is replaced by $\ellk(T^*)\to\ell(T^*)$) deducing the existence of a time $\hat{T}>T^*$ for which \eqref{thesis} holds true. This is absurd being $T^*$ the supremum, so we conclude.
	\end{proof}
	\begin{rmk}
		Since $\elldk(t)<1$ for a.e. $t\in[0,+\infty)$, by \eqref{thesis} we actually deduce that for every $p\geq 1$ it holds $\elldk\to\elld$ in $L^p(0,T)$ for every $T>0$. However this convergence cannot be improved to the case $p=+\infty$. Indeed let us consider $\ellzk=\ellz=1$, $\nuk=\nu=2$, $w^k\equiv w\equiv 0$ in $[0,+\infty)$, $\kappak\equiv\kappa\equiv 1/2$ in $[\ellz,+\infty)$, $u_0^k\equiv u_0\equiv u_1\equiv 0$ and $u_1^k(x)=3\chi_{[1-1/k,1]} (x)$ in $[0,1]$, so that $u_1^k\to 0$ in $L^2(0,1)$ but not in $L^\infty(0,1)$. Under these assumptions we have $(v,\ell)\equiv(0,1)$, so by Theorem~\ref{finalthm} we know that $\vk\to 0$ uniformly in $[0,T]\times[0,+\infty)$ for every $T>0$. This means that for every $k$ large enough there exists a small time $T_k>0$ such that for a.e. $t\in(0,T_k)$ we have:
		\begingroup
		\allowdisplaybreaks
		\begin{align*}
			\elldk(t)&=\max\left\{0,\frac{\left[u_1^k(\ellk(t){-}t))+\int_{0}^{t}\vk(\tau,\tau{-}t{+}\ellk(t))\d\tau\right]^2-e^{2t}}{\left[u_1^k(\ellk(t){-}t))+\int_{0}^{t}\vk(\tau,\tau{-}t{+}\ellk(t))\d\tau\right]^2+e^{2t}} \right\}\\
			&=\max\left\{0,\frac{\left[3+\int_{0}^{t}\vk(\tau,\tau{-}t{+}\ellk(t))\d\tau\right]^2-e^{2t}}{\left[3+\int_{0}^{t}\vk(\tau,\tau{-}t{+}\ellk(t))\d\tau\right]^2+e^{2t}} \right\}\\
			&\geq\frac{[3-1]^2-e}{[3+1]^2+e}=\frac{4-e}{16+e}>0,
		\end{align*}
		\endgroup
		and so $\elldk$ does not converge to $\elld\equiv 0$ in $L^\infty(0,T)$ for any $T>0$.
	\end{rmk}
	\begin{rmk}[\textbf{Presence of a forcing term}]
		If in the debonding model we take into account the presence of an external force $f$, then the equation the vertical displacement $u$ has to satisfy becomes:
		\begin{equation*}
			u_{tt}(t,x)-u_{xx}(t,x)+\nu u_t(t,x)=f(t,x), \quad t > 0 \,,\, 0<x<\ell(t),
		\end{equation*}
		while the energy-dissipation balance reads as:
		\begin{equation*}
			\mc E(t)+\mc A(t)+\int_{\ell_0}^{\ell(t)} \kappa(x) \d x=\mc E(0)+\mc W(t)+\mc F(t),\quad\quad\text{for every }t\in[0,+\infty),
		\end{equation*}
		where $\mc F(t):=\displaystyle \int_{0}^{t}\int_{0}^{\ell(\tau)}f(\tau,\sigma)u_t(\tau,\sigma)\d \sigma\d\tau$. In \cite{RivNar}, Remark~4.12, the authors proved that if the forcing term satisfies:
		\begin{equation}\label{forcing}
			f\in L^\infty_{\textnormal{loc}}((0,+\infty)^2)\quad\text{such that}\quad f\in L^\infty((0,T)^2)\quad\text{for every }T>0,
		\end{equation} 
		then Theorem~\ref{exuniq} still holds true, namely the coupled problem admits a unique solution $(u,\ell)$.\par 
		If now we consider, besides all the assumptions given in Subsection \ref{hypotheses}, a sequence of functions $\{f^k\}_{k\in\enne}$ satisfying \eqref{forcing} and we assume that:
		\begin{equation}\label{convf}
			f^k\to f\quad\mbox{in } L^\infty((0,T)^2), \quad \mbox{for every }T>0,
		\end{equation}
		then we can repeat all the proofs of the paper, obtaining even in this case the continuous dependence result \eqref{thesis} stated in Theorem~\ref{finalthm}. Indeed in this case the representation formula for the auxiliary function $\vk$, fixed $T<\frac\ellz 2$, reads as:
		\begin{equation}\label{new1}
			\vk(t,x)=\Ak(t,x)+\frac{(\nuk)^2}{8}\Hk(t,x)+\frac 12\iint_{R^k(t,x)}g^k(\tau,\sigma)\d\sigma\d\tau, \quad\mbox{ for every }(t,x)\in \overline{\Omega^k_T},
		\end{equation}	
		where $g^k(t,x):=e^{\nuk t/2}f^k(t,x)$. As a byproduct we obtain that for a.e. $y\in [-\ellzk,\varphik(T)]$ the function $\Theta_{\vk,\lambdak}^k$ introduced in \eqref{Lambdak} becomes:
		\begin{equation}\label{new2}
		\Theta_{\vk,\lambdak}^k(y)=\frac{\left[\vzdk(-y)-\vuk(-y)-\frac{(\nuk)^2}{4}\int_{0}^{\lambdak(y)}\vk(\tau,\tau{-}y)\d\tau-\int_{0}^{\lambdak(y)}g^k(\tau,\tau{-}y)\d\tau\right]^2}{2e^{\nuk\lambdak(y)}\kappak(\lambdak(y){-}y)}.
		\end{equation} 
		Using \eqref{new1}, \eqref{new2} and exploiting \eqref{convf} one can perform again the proofs of Sections \ref{sec2} and \ref{sec3}, concluding that Theorem~\ref{finalthm} still holds true even in this case.
	\end{rmk}

	\bigskip
	
	\noindent\textbf{Acknowledgements.}
	The author wishes to thank Prof. Gianni Dal Maso for many helpful discussions on the topic. The author is a member of the Gruppo Nazionale per l'Analisi Matematica, la Probabilit\`a e le loro Applicazioni (GNAMPA) of the Istituto Nazionale di Alta Matematica (INdAM).

	\bigskip
	
	\bibliographystyle{siam}

	{\small
		\vspace{15pt} (Filippo Riva) SISSA, \textsc{Via Bonomea, 265, 34136, Trieste, Italy}
		\par 
		\textit{e-mail address}: \textsf{firiva@sissa.it}
		\par
	}

\end{document}